\documentclass[12pt]{amsart}
\usepackage[T2A]{fontenc}
\usepackage[cp1251]{inputenc}

\usepackage{amsthm}
\usepackage{verbatim}
\usepackage{graphicx}
\usepackage{xcolor}
\usepackage{enumerate}
\usepackage{textcomp}
\usepackage{latexsym}
\usepackage{verbatim}
\usepackage[mathscr]{eucal}

\usepackage[pdftex,unicode,colorlinks,linkcolor=blue,
citecolor=red,bookmarksopen,pdfhighlight=/N]{hyperref}

\usepackage{textcomp}
\usepackage{upgreek}

\theoremstyle{plain} 
\newtheorem{theorem}{Theorem}
\newtheorem{proposition}{Proposition}
\newtheorem{lemma}{Lemma}

\newtheorem{dualtheorem}{Duality Theorem}

\newtheorem{maintheo}{Main Theorem} 
\newtheorem{mainlemma}{Main Lemma}

\newtheorem{PJf}{Classical Poisson\,--\,Jensen formula}
\newtheorem{SymPJf}{Symmetrization of the classical Poisson\,--\,Jensen formula}
\newtheorem{PJfASJ}{Poisson\,--\,Jensen formula for Arens\,--\,Singer and Jensen  measures}

\newtheorem{PJfASJV}{Poisson\,--\,Jensen formula for Arens\,--\,Singer and Jensen  potentials}

\theoremstyle{definition}
\newtheorem{remark}{Remark}

\usepackage{latexsym,amssymb,amsmath}
\usepackage{amsfonts,mathrsfs,amscd}
\usepackage{verbatim}
\usepackage[mathscr]{eucal}

\newcommand{\RR}{\mathbb{R}} 
 
\newcommand{\NN}{\mathbb{N}}

\newcommand{\pt}{{\rm pt}} 
\newcommand{\dd}{\,{\rm d}}
\DeclareMathOperator{\infill}{\rm in-fill}

\DeclareMathOperator{\clos}{clos} 
\DeclareMathOperator{\Int}{int}
\DeclareMathOperator{\Meas}{Meas}

\DeclareMathOperator{\har}{har}

\DeclareMathOperator{\comp}{cmp}
 
\DeclareMathOperator{\sbh}{sbh}

\DeclareMathOperator{\supp}{supp} 

\DeclareMathOperator{\sgn}{sgn} 
\DeclareMathOperator{\Borel}{Bor}

\DeclareMathOperator{\Dom}{Dom}

\DeclareMathOperator{\Conn}{Conn}

\title{Poisson\,--\,Jensen Formulas and Balayage of Measures}
\author{Bulat Khabibullin}
\email{khabib-bulat@mail.ru}

\begin{document}

\maketitle

\begin{abstract}
Our main results are certain developments of the classical Poisson\,--\,Jensen formula for subharmonic functions. The basis of the classical Poisson\,--\,Jensen formula is the natural duality between harmonic measures and Green's functions. Our generalizations use some  duality between the balayage of measures and and their potentials. 
\end{abstract}
\keywords{subharmonic function, balayage, potential, Riesz measure, Jensen measure}

\tableofcontents

\section{Introduction}\label{Int}

\subsection{\sf On the classical Poisson-Jensen formula}

Let $D$\/ be a {\it bounded domain\/} in the $d$-dimensional Euclidean space  $\RR^d$ with the {\it closure\/} $\clos D$ in $\RR^d$ and the {\it boundary\/} $\partial D$ in $\RR^d$. Then, for any $x\in D$  there are the {\it extended harmonic measure $\omega_D(x, \cdot)$ for  $D$ at $x\in D$\/} 
as a Borel probability measure on $\RR^d$ with support on
$\partial D$ and the {\it generalized Green's function $g_D(\cdot , x )$  for $D$ with pole at $x$\/}  extended by zero values on  the complement $\RR^d\!\setminus \!\clos D$ and by the upper semicontinuous regularization on $\partial D$ from $D$ \cite{HK}, \cite{AG}, \cite{R}, \cite{Helms}, \cite{Doob}, \cite{Landkoff} (see also \eqref{oB} and \eqref{gD} in Subsec.~\ref{ccPJ} below). 

Let $u\not\equiv -\infty$ be  a {\it subharmonic function on\/}  $\clos D$, i.e., on an open set containing $\clos D$,  with its {\it Riesz measure\/} $\varDelta_u$ on this open set (see in detail \S\S~\ref{Sssm}--\ref{Sssf} and \eqref{df:cm}).

\begin{PJf}[{\rm \cite[Theorem 5.27]{HK}, \cite[4.5]{R}}] 
\begin{equation}\label{clasPJ}
u(x)=\int_{\partial D} u\dd \omega_D(x,\cdot)-\int_{\clos D} g_D(\cdot,x)\dd \varDelta_u \quad\text{for each $x\in D$.} 
\end{equation}
\end{PJf}

 For $s\in \RR$, we set  
\begin{subequations}\label{kK}
\begin{align}
k_s(t)& := \begin{cases}
\ln t  &\text{ if $s=0$},\\
 -\sgn (s)  t^{-s} &\text{ if $s\in \RR\!\setminus\!0$,} 
\end{cases}
\qquad  t\in \RR^+\!\setminus\!0,
\tag{\ref{kK}k}\label{{kK}k}
\\
K_{d-2}(y,x)&:=\begin{cases}
k_{d-2}\bigl(|y-x|\bigr)  &\text{ if $y\neq x$},\\
 -\infty &\text{ if $y=x$ and $d\geq 2$},\\
0 &\text{ if $y=x$ and  $d=1$},\\
\end{cases}
\quad  (y,x) \in \RR^d\times \RR^d.
\tag{\ref{kK}K}\label{{kK}K}
\end{align}
\end{subequations}
The following functions 
\begin{equation}\label{pqKo}
p\colon y\underset{\text{\tiny $y\in \RR^d$}}{\longmapsto}  g_D(y,x)+K_{d-2}(y,x), 
\quad   q\colon y\underset{\text{\tiny $y\in \RR^d$}}{\longmapsto} K_{d-2}(y,x)  
\end{equation}
are subharmonic with Riesz probability measures 
$\varDelta_p=\omega_D(x, \cdot)$ and $\varDelta_q=\delta_x$,
where $\delta_x$ is the  Dirac measure at $x\in D$: $\delta_x\bigl(\{x\}\bigr)=1$.
The following symmetric equivalent form of the classical Poisson\,--\,Jensen formula \eqref{clasPJ}  immediately follows  from the suitable definitions of  harmonic measures and Green's functions and is briefly discussed in  Subsec.~\ref{ccPJ}. 
\begin{SymPJf} 
If we choose $p, q$ as in \eqref{pqKo} and put $S=\clos D$, then   \eqref{clasPJ} can be rewritten in the symmetric form
\begin{equation}\label{1_0}
\int_S u\dd \varDelta_q+ \int_S p\dd \varDelta_u 
=\int_S u\dd \varDelta_p +\int_S q\dd \varDelta_u. 
\end{equation} 
\end{SymPJf}
Equality  \eqref{1_0}   reflects the fact that the {\it Laplace operator\/} $\bigtriangleup$ is {\it self-adjoint\/} for some formal bilinear integral form
 $(u,\bigtriangleup w):=\int u\bigtriangleup\! w=\int {\bigtriangleup u}\, w=(\bigtriangleup u,w)$ , where $w:=q-p$.

The following result  is a special case of our Main Theorem  from  Subsec.~\ref{SsPJ}, but already significantly develops the classical Poisson\,--\,Jensen formula \eqref{pqKo}--\eqref{1_0}.  

\begin{theorem}\label{thpq} Let  $S\subset \RR^d$ be a non-empty compact set, and     $p\not\equiv -\infty$ and $q\not\equiv -\infty$ be a pair  of subharmonic functions on\/ $S$ with Riesz measures $\varDelta_p$ and $\varDelta_q$, respectively. If   $p$ and $q$ are harmonic outside $S$ and $p=q$ outside\/  $S$, then the symmetric  Poisson\,--\,Jensen formula  \eqref{1_0}  holds for each subharmonic function $u\not\equiv -\infty$ on $S$ with Riesz measure $\varDelta_u$.
\end{theorem}

Our Main Lemma is formulated in Subsec.~\ref{SsPJm} and gives a symmetric Poisson\,--\,Jensen formula for measures and their potentials. The Main Lemma is proved in Sec.~\ref{PML}. The proof of the Main Lemma use Theorem \ref{lemPQ} 
on representations for pairs of subharmonic functions.   
Theorem \ref{lemPQ}  from Subsec.~\ref{RSF}  is also of independent interest.
 
The Main Theorem is formulated in Subsec.~\ref{SsPJ} and gives a full symmetric Poisson\,--\,Jensen formula for subharmonic integrands. The proof of 
the Main Theorem  in Sec.~\ref{PMT} essentially uses the Main Lemma.  
Theorem \ref{thpq} is deduced from the Main Theorem in Subsec.~\ref{SsPJ}.
The next Subsec.~\ref{ccPJ} contains a discussion of the classical symmetric  Poisson\,--\,Jensen formula  \eqref{pqKo}--\eqref{1_0}  as a consequence of Theorem \ref{thpq}. 

Our Duality Theorems \ref{DT1}--\ref{DT3} in Sec.~\ref{DT} give a complete description of potentials of measures obtained as a certain process of balayage of measures with compact support. In order to prove Duality Theorems \ref{DT1}--\ref{DT3}, we use both the  Main Lemma and the Main Theorem.

We proceed to precise and detailed definitions and formulations.

\subsection{\sf Basic notation, definitions, and conventions}\label{Ss11}

 The reader can skip this Subsec. \ref{Ss11}
and return to it only if necessary.

\subsubsection{\tt \,Sets, topology, order}
We denote by $\NN:=\{1,2,\dots\}$, $\RR$, and $\RR^+:=\{x\in \RR\colon x\geq 0\}$  the sets of {\it natural,\/} of {\it real,\/} and  of {\it positive\/} numbers, each endowed with its natural order ($\leq$, $\sup/\inf$), algebraic, geometric  and topological structure.  We denote singleton sets by a symbol without curly brackets. So, $\NN_0:=\{0\}\cup \NN=:0\cup \NN$, and  $\RR^+\!\setminus\!0:=\RR^+\!\setminus\!\{0\}$ is the set of {\it strictly positive\/} numbers, etc. The {\it extended real line\/}  $\overline \RR:=-\infty\sqcup\RR\sqcup+\infty$ is the order completion of $\RR$ by the  {\it disjoint union\/} $\sqcup$  with $+\infty:=\sup \RR$ and $-\infty:=\inf \RR$ equipped with the order topology with two  ends $\pm\infty$, $\overline \RR^+:=\RR^+\sqcup+\infty$;  $\inf \varnothing :=+\infty$, $\sup \varnothing :=-\infty$ for the {\it empty set\/} $\varnothing$ etc. 
The same symbol $0$ is also used, depending on the context, to denote  zero vector, zero function, zero measure, etc.

We denote by $\RR^d$ the  {\it Euclidean space of $d\in \NN$ dimensions\/}  with the  {\it Euclidean norm\/} $|x|:=\sqrt{x_1^2+\dots+x_d^2}$ of $x=(x_1,\dots ,x_d)\in \RR^d$,  and by $\RR^d_{\infty}:=\RR^d\sqcup\infty$ we denote  
 the {\it  Alexandroff\/} (Aleksandrov) 
{\it one-point compactification\/}   of $\RR^d$
obtained by adding one extra point $\infty$. For a subset $S\subset \RR^d_{\infty}$ or a subset $S\subset \RR^d$ we let $\complement S :=\RR^d_{\infty}\!\setminus\!S$, $\clos S$, $\Int S:=\complement (\clos \complement  S)$, and $\partial S:=\clos S\!\setminus\!\Int S$ denote its
 {\it complement,\/} {\it closure,} {\it interior,} and {\it boundary\/}  always in $\RR^d_{\infty}$, and $S$ is equipped with the topology induced from $\RR^d_{\infty}$. If $S'$ is a relative compact subset in $S$, i.e., $\clos S'\subset S$,  then we write $S'\Subset S$.  We denote by 
 $B(x,t):=\{y\in \RR^d\colon |y-x|< t\}$, $\overline B(x,t):=\{y\in \RR^d\colon |y-x|\leq  t\}$, $\partial \overline B(x,t):=\overline B(x,t)\!\setminus\!  B(x,t)$  an {\it open ball, closed ball,\/} a {\it circle of radius $t\in \RR^+$ centered at $x\in \RR^d$}, respectively. 

Let $T$ be a topological space, and $S$ be a subset in $T$.   
We denote by $\Conn_T S$ or $\Conn_T (S)$ the  set of all connected components of $S\subset T$ in $T$. 

\underline{Throughout this paper} $O\neq \varnothing$ will denote  an  {\it open subset  in\/ $\RR^d$},   and $D\neq \varnothing$ is a  {\it domain in\/ $\RR^d$,\/}  i.e., an open connected subset in $\RR^d$. 

\subsubsection{\tt \,Measures and charges}\label{Sssm}

The convex cone over $\RR^+$ of all Borel, or Radon,  positive measures $\mu\geq 0$  on the $\sigma$-algebra $\Borel (S)$ of all {\it Borel subsets\/} of $S$ is denoted by $\Meas^+(S)$; $\Meas^+_{\comp}(S)\subset \Meas^+(S)$ is the subcone of $\mu\in \Meas^+(S)$ with compact  {\it support\/} $\supp \mu$ in $S$, $\Meas(S):=\Meas^+(S)-\Meas^+(S)$ is the vector lattice over $\RR$ of {\it charges,\/} or signed measures, on $S$, $\Meas^{+1}(S)
$ is the convex set of {\it probability\/} measures on $S$, 
$\Meas_{\comp}^{1+}(S):=\Meas^{1+}(S)\cap \Meas_{\comp}(S)$,
and   $\Meas_{\comp}(S):=\Meas^+_{\comp}(S)-\Meas^+_{\comp}(S)$.
 For a charge $\mu \in \Meas(S)$, we let
$\mu^+:=\sup\{0,\mu\}$, $\mu^-:=(-\mu)^+$ and $\mu:= \mu^++\mu^-$ respectively denote its {\it upper,} {\it lower,} and {\it total variations,} and $\mu(x,t):=\mu\bigl( \overline B(x,t)\bigr)$. 

For an {\it extended numerical  function\/} $f\colon S\to \overline \RR$ we allow values $\pm\infty$ for Lebesgue integrals \cite[Ch.~3, Definiftion 3.3.2]{HK} (see also \cite{Bourbaki})
\begin{equation}\label{int}
\int_Sf\dd \mu\in \overline \RR, \quad \mu \in \Meas^+(S),
\end{equation}
and we say that $f$ is {\it $\mu$-integrable on\/} $S$ if  the integral in \eqref{int} is finite.

\subsubsection{\tt \,Subharmonic functions}\label{Sssf}

We denote  by $\sbh (O)$  the convex cone over $\RR^+$  of all   {\it subharmonic\/} (locally convex if $d = 1$) functions on $O$, including functions that are identically equal to $-\infty$ on some components $C\in \Conn_{\RR^d_{\infty}}(O)$. Thus, $\har(O):=\sbh(O)\cap \bigl(-\sbh(O)\bigr)$
is the vector space over $\RR$ of all {\it harmonic\/} (locally affine if $d = 1$) functions on $O$.
Each function \begin{equation*}
u\in \sbh_*(  O):=\bigl\{u\in \sbh( O)\colon u\not\equiv-\infty 
\text{ on each }C\in \Conn_{\RR^d_{\infty}}(O)\bigr\}
\end{equation*}
is associated with its {\it Riesz measure\/}
\begin{equation}\label{df:cm}
\varDelta_u:= c_d {\bigtriangleup}  u\in \Meas^+( O), 
\quad c_d:=\frac{\Gamma(d/2)}{2\pi^{d/2}\max \{1, d-2\bigr\}}, 
\end{equation}
where ${\bigtriangleup}$  is  the {\it Laplace operator\/}  acting in the sense of the  theory of distribution or generalized functions, and 
$\Gamma$ is the gamma function. If $u\equiv -\infty$ on $C\in \Conn_{\RR^d_{\infty}}(O)$, then  we set $\varDelta_{-\infty}(S):=+\infty$ for each  $S\subset C$. Given $S\subset \RR^d$, we set 
\begin{equation*}
\begin{split}
\text{Sbh}(S)&:=\bigcup \Bigl\{\sbh(O')\colon 
S\subset O'\overset{\text{\tiny open}}{=}\Int O'\subset \RR^d\Bigr\},
\\
\text{Sbh}_*(S)&:=\bigcup \Bigl\{\sbh_*(O')\colon 
S\subset O'\overset{\text{\tiny open}}{=}\Int O'\subset \RR^d\Bigr\},
\\
\text{Har}_*(S)&:=\bigcup \Bigl\{\har(O')\colon 
S\subset O'\overset{\text{\tiny open}}{=}\Int O'\subset \RR^d\Bigr\}.
\end{split}
\end{equation*}
Consider a  binary relation $\cong\,\subset  \text{Sbh}(S)\times \text{Sbh}(S)$
on $\text{Sbh}(S)$ defined by the rule: $U\cong V$ if {\it there is an open set $O'\supset S$ in\/ $\RR^d$ such that $U\in \sbh (O')$, $V\in \sbh(O')$, and $U(x)=V(x)$ for each $x\in O'$.}  This relation  $\cong$ is an {\it equivalence relation\/} on  $\text{Sbh}(S)$, on $\text{Sbh}_*(S)$, and on $\text{Har}(S)$.
The quotient sets of $\text{Sbh}(S)$, of $\text{Sbh}_*(S)$, and of $\text{Har}(S)$ 
by $\cong$  are denoted below by 
$\sbh(S):=\text{Sbh}(S)/\cong$,  $\sbh_*(S):=\text{Sbh}_*(S)/\cong$,
and $\har(S):=\text{Har}(S)/\cong$,  respectively.  The equivalence class $[u]$ of $u$ is denoted without square brackets as simply $u$, and we do not distinguish between the equivalence class $[u]$ and the function $u$ when possible. 
So, for $u,v\in \sbh(S)$, we write {\it``$u=v$ on $S$''} if $[u]=[v]$ in $\sbh(S):=\text{Sbh}(S)/\cong$, or, equivalently, $u\cong v$ on $\text{Sbh}(S)$,
and  we write $u\not\equiv-\infty$ if $u\in \sbh_*(S)$.
The concept of the Riesz measure $\varDelta_u$ of $u\in \sbh(S)$ is correctly and uniquely defined by the restriction $\varDelta_u\bigm|_S$ of the Riesz measure $\varDelta_u$ to $S$.  For $u\in \sbh(S)$ and $v\in \sbh(S)$, the concepts {\it ``$u\leq v$ on $S$'',\/ {\rm and}  ``$u=v$ outside $S$'', ``$u\leq v$ outside $S$'', ``$u$ is harmonic outside S'',\/} etc. defined naturally:
$u(x)\leq v(x)$ {\it for each} $x\in S$, and  {\it there exits an open set 
$O'\supset S$} {\rm such that } $u(x)=v(x)$ {\it for each $x\in O'\!\setminus\!S$}, $u(x)\leq v(x)$ {\it for each $x\in O'\!\setminus\!S$}, {\it the restriction $u\bigm|_{O'\!\setminus\!S}$ is harmonic on 
$O'\!\setminus\!S$,} respectively. 
So, Theorem \ref{thpq} from Introduction is formulated precisely in this interpretation.

\subsubsection{\tt \,Balayage}\label{Bal}

Let $S\in \Borel( \RR^d)$ and $H$\/ be a set of {\it upper semicontinuous\/} functions  
$f\colon S\to \overline \RR\!\setminus\!+\infty$. 
A  measure $\upomega\in \Meas^+_{\comp}(S)$ is called the {\it balayage\/} of a measure $\varDelta \in \Meas^+_{\comp}(S)$ for $S$ with respect to $H$  \cite{Meyer}, \cite{BH}, \cite[Definition 5.2]{KhaRozKha19}, or, briefly, {\it $\upomega$ is a\/ $H$-balayage of\/ $\varDelta$,} and we write $\varDelta \preceq_{H}\upomega$ or $\upomega\succeq_H\varDelta$ if
\begin{equation}\label{bhar0}
\int_S h\dd \varDelta \leq \int_S h\dd \upomega \quad\text{\it for each $h\in H$ in accordance with \eqref{int}}. 
\end{equation}
If $\varDelta \preceq_{H}\upomega$ and $\upomega \preceq_{H}\varDelta$, then we write
$\varDelta\simeq_{H}\upomega$. 
The following properties are obvious:
\begin{enumerate}[{1.}]
\item The binary relation $\preceq_{H}$ (respectively $\simeq_{H}$) on $\Meas^+_{\comp}(S)$ 
 is  a  {\it preorder,\/} i.e., a reflexive and transitive relation, 
(respectively, an {\it equivalence\/}) on  $\Meas^+_{\comp}(S)$.
\item If $H$ contains a strictly positive (respectively, negative) constant, then $\varDelta (S)\leq \upomega(S)$ ($\varDelta (S)\geq \upomega(S)$, respectively). 
\item If $H=-H$, then the order  $\preceq_{H}$ is the equivalence $\simeq_{H}$. 
So, if $H=\har(S)$, then   $\upomega $ is a $\har(S)$-balayage of 
$\varDelta $ if and only if $\varDelta\simeq_{\har(S)}\upomega$, i.e.,
\begin{equation}\label{bhar}
\int_S h\dd \varDelta =\int_S h\dd \upomega \quad\text{\it for each $h\in \har(S)$} \quad\text{\it and} \quad \varDelta(S)=\upomega(S). 
\end{equation}

\item\label{4} If $\varDelta \preceq_{\sbh(S)}\upomega$, then $\varDelta \preceq_{\har(S)}\upomega$. The converse is not true
\cite[XIB2]{Koosis}, \cite[Example]{MenKha19}. 
\item\label{pr:diff} If $\upomega \in \Meas^+_{\comp}(O)$   is a 
$\bigl(\sbh(O)\cap C^{\infty}(O)\bigr)$-balayage of $\varDelta\in \Meas_{\comp}^+(O)$, where $C^{\infty}(O)$ is the class 
of all infinitely differentiable functions on $O$, then  $\varDelta\preceq_{\sbh(O)}\upomega$, since for each function $u\in \sbh(O)$ there exists a sequence of functions $u_j\underset{\text{\tiny $j\in \NN$}}{\in} \sbh(O)\cap C^{\infty}(O)$ decreasing to it  \cite[Ch. 4, 10, Approximation Theorem]{Doob}.
\end{enumerate}

\begin{remark} Balayage of  charges and  measures   with a non-compact support is also occur frequently and are used  in Analysis. So, a bounded domain $D\subset \RR^d$ is called a {\it quadrature domain\/} (for harmonic functions)  if there is a charge $\varDelta \in \Meas_{\comp} (D)$ such that 
the restriction of the Lebesgue measure $\lambda$ to $D$  is a balayage of 
$\varDelta$ with respect to the class of all harmonic  $\lambda$-integrable functions on $D$. In connection with the quadrature domains, see very informative overview  \cite[3]{quad} and bibliography in it.
\end{remark}

\subsubsection{\tt \,Potentials}\label{Pjhar}
For a {\it charge\/} $\mu\in \Meas_{\comp}(O)$ its  {\it potential\/}  
\begin{equation}\label{pot} 
\pt_{\mu}\colon \RR^d_{\infty}\to \overline \RR, \quad \pt_{\mu}(y)
\overset{\eqref{kK}}{:=}\int_O K_{d-2}(x,y) \dd \mu (x), 
\end{equation}
is uniquely determined on \cite{Arsove}, \cite[3.1]{KhaRoz18}
\begin{equation}\label{Dom}
\Dom \pt_{\mu} 
:=\left\{y\in \RR^d\colon
\inf\left\{ \int_{0}^1\frac{\mu^-(y,t)}{t^{d-1}} \dd t, \int_{0}^1\frac{\mu^+(y,t)}{t^{d-1}} \dd t\right\}<+ \infty 
\right\}
\end{equation}
by values in  $\overline \RR$, 
and  the set $E:=(\complement \Dom \pt_{\mu})\!\setminus\!\infty$ is {\it polar\/} with zero {\it outer capacity\/} 
\begin{equation*}
\text{Cap}^*(E):=\inf_{E\subset O'\overset{\text{\tiny open}}{=}\Int O'}  
\sup_{\stackrel{C\overset{\text{\tiny closed}}{=}\clos C\overset{\text{\tiny compact}}{\Subset} O}{\nu\in \Meas^{1+}(C)}} 
 k_{d-2}^{-1}\left(\iint K_{d-2} (x,y)\dd \nu (x) \dd \nu(y) \right).
\end{equation*}
Evidently $\pt_{\mu}\in \har\bigl(\RR^d\!\setminus\!\supp |\mu|\bigr)$, and if $\mu \in \Meas^+_{\comp}(\RR^d)$, then $\pt_{\mu}\in \sbh_*(\RR^d)$. 

\subsubsection{\tt \,Inward filling of sets with respect to an open set}

Let $O$ be an open set in $\RR^d$. The union of $S\subset O$ with all components $C\in \Conn_O(O\!\setminus\!S)$ such that $C\Subset O$  will be called the {\it inward filling\/} of $S$ with respect to $O$, and we denote this union by $\infill_O S$ or $\infill_O (S)$, although in \cite[1.7]{Gardiner}, \cite{BagbyGauther}, \cite[Sec.~12]{Gauther_B}, \cite[\S~1]{Kha03} another  notation $\hat S$ was used. 
Denote by $O_{\infty}$  the {\it  Alexandroff one-point compactification\/} of  $O$ with underlying set $O \sqcup \infty$, where extra point $\infty \not\in O$  can be identified with the boundary $\partial O$ or the complement $\complement O$, considered as a single point.
The following  elementary properties of the  inward filling  will  often be used without mentioning them.

\begin{proposition}[{\rm \cite[6.3]{Gardiner}, \cite{Gauther_B}, \cite{BagbyGauther}, \cite{Gauther_C}}]\label{KOc}
Let  $S$ be a compact  set in an open set  $O \subset \RR^d$.  Then 
\begin{enumerate}[{\rm (i)}]
\item\label{Ki} $\infill_{O}  S$ is a compact subset in $O$, and
$\infill_{O} \bigl( \infill_{O}  S\bigr)=\infill_{O} S$;
\item\label{Kii} the set\/ $O_{\infty}\!\setminus\!\infill_{O}  S$ is connected and locally connected subset in\/ $O_{\infty}$; 
\item\label{Kiii}   the inward filling of $S$ with respect to $O$ coincides with the complement in $O_{\infty}$ of  connected component of $O_{\infty}\!\setminus\!S$ containing the point $\infty$, i.\,e., 
$\infill_{O}  S=O_{\infty}\!\setminus\! C_{\infty}$, 
where $\infty\in C_{\infty}\in \Conn_{O_{\infty}}(O_{\infty}\!\setminus\! S)$;
\item\label{Kiv} if $O'\subset \RR_{\infty}^d$ is an open subset and 
 $O\subset  O'$, then  $\infill_{O}  S\subset \infill_{O'}  S$;
\item\label{Kv} if $S'\subset S$ is a compact  subset in 
 $O$, then  $\infill_{O}  S'\subset \infill_O  S$;
\item\label{Kvi} $\RR^d\!\setminus\!\infill_{O} S$ has only finitely many components, i.\,e., $\#\Conn_{\RR^d_{\infty}}(\RR^d\!\setminus\!\infill_{O}  S)<\infty$. 
\end{enumerate}
\end{proposition}

\section{Poisson\,--\,Jensen formulas}\label{SshPJ}
\setcounter{equation}{0}

\subsection{\sf Main result  for measures and their potentials}\label{SsPJm}

\begin{mainlemma}\label{th1}  
Let $\varDelta \in \Meas_{\comp}^+(O)$, $\upomega\in \Meas_{\comp}^+(O)$, and 
\begin{equation}\label{SO}
S_O:=\infill_O (\supp \varDelta \cup \supp \upomega).
\end{equation}
The following seven  statements are equivalent:
\begin{enumerate}[{\rm I.}]
\item\label{Im} $\varDelta\preceq_{\har(O)}\upomega$. 
\item\label{IIm} $\varDelta\simeq_{\har(S_O)}\upomega$. 
\item\label{IIIm}  $\pt_{\varDelta}=\pt_{\upomega}$ on\/ $\RR^d\!\setminus\!S_O$. 
\item\label{IVm} There are a compact subset $S$ in $O$, a function $q\in \sbh_*(S)$ with Riesz measure $\varDelta_q=\varDelta$, and a function $p\in \sbh_*(S)$ with Riesz measure $\varDelta_p
=\upomega$ such that $q$ and $p$ are harmonic outside $S$, and 
$q=p$ outside $S$.  
\item\label{Vm} The 
symmetric Poisson\,--\,Jensen formula for measures and their potentials is valid:
\begin{subequations}\label{fPJ}
\begin{align}
\int u\dd \varDelta+\int_B {\pt}_{\upomega} \dd \varDelta_u&=
 \int u\dd \upomega +\int_B {\pt}_{\varDelta} \dd \varDelta_u
\tag{\ref{fPJ}f}\label{{fPJ}f}\\
\text{for each $B\in \Borel(\RR^d)$ such that \/}& S_O\subset B\Subset O
\text{ and  for each $u\in \sbh_*(\clos B)$}.
\tag{\ref{fPJ}B}\label{{fPJ}B}
\end{align}
\end{subequations}
\item\label{VIm} For each  $q\in \sbh_*(S_O)$ with  $\varDelta_q=\varDelta$ there is  $p\in \sbh_*(S_O)$ with  $\varDelta_p=\upomega$ such that 
\begin{equation}\label{2R}
\int u\dd \varDelta+\int_{S_O} p\dd \varDelta_u =\int u\dd \upomega+\int_{S_O}
q\dd \varDelta_u\quad\text{for each $u\in \sbh_*(O)$}. 
\end{equation}
\item\label{VIIm} There are a compact subset $S\supset S_O$
in $O$ and a pair of functions  $q\in \sbh_*(S_O)$ and $p\in \sbh_*(S_O)$ with  Riesz measures $\varDelta_q=\varDelta$ and   $\varDelta_p=\upomega$, respectively,  such that\/ the equality in \eqref{2R} is fulfilled  for each special  subharmonic function $u_x\colon y\underset{\text{\tiny $y\in \RR^d$}}{\longmapsto} K_{d-2}(y,x)$ with $x\in O\!\setminus\!S$
 instead of all functions   $u\in \sbh_*(O)$ in\/ \eqref{2R}.
\end{enumerate} 
 
\end{mainlemma}
The proof of the Main Lemma will be given only after some preparation in Section  \ref{PML}. 

\subsection{\sf Full version of the Poisson\,--\,Jensen formula for subharmonic integrands}\label{SsPJ}

The starting point of the  Main Lemma is a pair of 
measures $\varDelta, \upomega \in \Meas_{\comp}^+(O)$. 
Our Main Theorem is a functional  counterpart of the Main Lemma. The starting point in it is now a pair of subharmonic functions  from \eqref{{Spq}s}.
\begin{maintheo}\label{MT}
Let 
\begin{subequations}\label{Spq}
\begin{align}
\varnothing \neq S\overset{\text{{\rm \tiny closed}}}{=}\clos S
\overset{\text{{\rm \tiny compact}}}{\Subset} O\overset{\text{{\rm \tiny open}}}{=}&\Int O\subset \RR^d, \quad  S_O:=\infill_{O}S,
\tag{\ref{Spq}S}\label{{Spq}S}\\
q\in \sbh_*(O)\cap \har (O\!\setminus\!S),& \quad 
p\in \sbh_*(O)\cap \har (O\!\setminus\!S),
\tag{\ref{Spq}s}\label{{Spq}s}\\
S_{\neq}:=\bigl\{x\in O\colon& q(x)\neq p(x)\bigr\}.
\tag{\ref{Spq}$\neq$}\label{{Spq}neq}
\end{align}
\end{subequations}
The following four  statements are equivalent:
\begin{enumerate}[{\rm [I]}]
\item\label{Is} $ S_{\neq}\subset S_O$, i.e.,
$q=p$ on $O\!\setminus\!S_O$.
\item\label{IIs} There is a function $h\in \har(O)$ such that 
\begin{equation}\label{pqh}
\begin{cases}
q=\pt_{\varDelta_q}+h\\
p=\pt_{\varDelta_p}+h
\end{cases}
\text{on $O$ and\/  }  \pt_{\varDelta_q}=\pt_{\varDelta_p} \text{ on\/ $\RR^d\!\setminus\!S_O$},
\end{equation} 
where $\varDelta_q\in \Meas^+(S)$ and $\varDelta_p\in \Meas^+(S)$ are the Riesz measures of $q$ and $p$.
\item\label{IIIs} The full symmetric Poisson\,--\,Jensen formula 
is valid: 
\begin{subequations}\label{2}
\begin{align}
\int_S u\dd \varDelta_q+\int_{B} p\dd \varDelta_u &=\int_S u\dd \varDelta_p+\int_{B} q\dd \varDelta_u
\quad \text{for each $B\in \Borel(\RR^d)$}
\tag{\ref{2}f}\label{{2}f}\\
\text{under } S_O\cap S_{\neq}&\subset B\Subset O
\text{ and  for each $u\in \sbh_*(S_O\cup \clos B )$}.
\tag{\ref{2}B}\label{{2}B}
\end{align}
\end{subequations}
\item\label{IVs}  \eqref{2} holds   
for a sequence  of sets $B_j\underset{\text{\tiny $j\in \NN_0$}}{\in} \Borel (O)$ such that $B_0:=S_O\underset{\text{\tiny $j\in \NN$}}{\subset} B_j\underset{\text{\tiny $j\in \NN$}}{\Subset} O$ and\/ $\bigcup_{j\in \NN}B_j=O$ instead of all $B\in \Borel(\RR^d)$ with $S_O\cap S_{\neq} \subset B\Subset O$ in\/ \eqref{2} and for each special  subharmonic function $u_x\colon y\underset{\text{\tiny $y\in \RR^d$}}{\longmapsto} K_{d-2}(y,x)$ with $x\in O\!\setminus\!B_0=O\!\setminus\!S_O$
 instead of all  functions $u\in \sbh_*(S_O\cup \clos B)$ in\/ \eqref{{2}B}.
\end{enumerate} 
\end{maintheo}

We can now prove Theorem \ref{thpq} of the Introduction.
\begin{proof}[of Theorem\/ {\rm \ref{thpq}}] 
There is an open set $O\subset \RR^d$ such that 
$u\in \sbh_*(O)$, $q\in \sbh_*(O)$ and $p\in \sbh_*(O)$ are harmonic on  $O\!\setminus\! S$, and  also $q=p$ on  $O\!\setminus\!S$. 
Evidently, in the notation \eqref{Spq}, we have $S_O\cap S_{\neq}\subset S
\subset S_O\Subset O$ and $u\in \sbh_*(S_O)$. Theorem \ref{thpq} with \eqref{1_0} follows from  implication [\ref{Is}]$\Rightarrow$[\ref{IIIs}] of the Main Theorem, since we can choose  
 $B:=S$ in \eqref{2}. 
\end{proof}

\subsection{\sf In detail on the classical Poisson\,--\,Jensen formula}\label{ccPJ}

If  $x\in D\Subset O$, then  
the {\it extended harmonic measure  $\omega_D(x, \cdot )\in \Meas^{1+}(\partial D)\subset \Meas^{1+}_{\comp}(\RR^d)$ (for $D$ at\/} $x$)  defined  on sets $B\in \Borel(\RR^d)$ by
\begin{equation}\label{oB}
\omega_D(x, B):=\sup\left\{u(x)\colon u\in \sbh(D),\; \limsup_{D\ni y'\to y\in \partial D } u(y')\leq 
\begin{cases}
1\text{ for $y\in B\cap \partial D$}\\
0\text{ for $y\notin B\cap \partial D$}
\end{cases} 
\right\}
\end{equation}
is a $\har(O)$-balayage of $\delta_x$ with obvious equalities 
\begin{equation*}
\infill\bigl(\supp \delta_x\cup \supp \omega_D(x, \cdot )\bigr) =\infill (x\cup \partial D) =\clos D,
\end{equation*}
the potential (see \cite[Ch.~4,\S~1,2]{Landkoff}) 
\begin{multline}\label{Dog}
\pt_{\omega_D(x, \cdot )-\delta_x}(y) =
\pt_{\omega_D(x, \cdot)}(y)-\pt_{\delta_x}(y) 
\\=\int_{\partial D} K_{d-2}(y,x') \dd_{x'} \omega_D(x, x')-K_{d-2}(y,x )
=g_D(y,x), \quad y\in \RR^d_{\infty}, \quad x\in D,  
\end{multline}
is equal to  the {\it generalized Green's function $g_D(\cdot,x)\colon \RR^d_{\infty}\to \overline \RR^+$ (for $D$ with pole at $x$ and  $g_D(x,x):=+\infty$)\/}
defined on $\RR^d_{\infty}\!\setminus\!x$ by upper semicontinuous regularization
 \begin{equation}\label{gD}
\begin{split}
g_D(y ,x)&:=\check{g}^*(y,x)
:=\limsup_{\RR^d\ni y'\to y} \check{g}(y',x)\in \overline{\RR}^+ \quad 
\text{for each $y\in \RR^d_{\infty}\!\setminus\!x$, where} \\
\check{g}(y,x)&:=\sup\left\{u(y)\colon
u\in \sbh(\RR^d\!\setminus\!x), \;
\begin{cases}
u(y')\leq 0 \text{ for each  }y\notin \clos D, \\
\limsup\limits_{x\neq y\to x}\dfrac{u(y)}{-K_{d-2}(x,y)}\leq 1
\end{cases} 
\right\}. 
\end{split}
\end{equation} 

The equalities \eqref{Dog} give \eqref{pqKo}
with $\varDelta_p=\omega_D(x, \cdot)$ and $\varDelta_q=\delta_x$. Thus, Theorem \ref{thpq} implies the symmetric Poisson\,--\,Jensen  formula  \eqref{1_0} which can be written in detail as
\begin{multline}\label{10a}
\int_{\clos D} u\dd \delta_x+ \int_{\clos D} 
\Bigl(g_D(\cdot,x)+K_{d-2}(\cdot,x)\Bigr)
\dd \varDelta_u 
\\=\int_{\clos D} u\dd \omega_D(x,\cdot)
 +\int_{\clos D} K_{d-2}(\cdot,x)\dd \varDelta_u. 
\end{multline}  
The latter coincides with the classical Poisson\,--\,Jensen  formula \eqref{clasPJ}.

\subsection{\sf The Poisson\,--\,Jensen formula for the Arens\,--\,Singer, and Jensen measures and potentials}\label{PJAS}

Our results presented in this Subsec.~\ref{PJAS}  are intermediate between the classical Poisson\,--\,Jensen formula \eqref{clasPJ} and the symmetric Poisson\,--\,Jensen formula \eqref{1_0} of Theorem \ref{thpq}. 

If $x\in O$ and   $\delta_x\preceq_{\har(O)}\upomega\in \Meas^+_{\comp}(O)$, then $\upomega$ is called an {\it Arens\,--\,Singer measure on $O$ at $x$\/}  \cite[Ch.~3]{Gamelin}, \cite{Sarason}, \cite{Kha96}, \cite[Definition 1]{Kha03}, \cite{Kha01II}, \cite{Kha07}, \cite{KudKha09}, or representing measure. We denote by $AS_x(O)\subset \Meas^{1+}_{\comp}(O)$ the class of all Arens\,--\,Singer measure on $O$ at $x$. 
If  $\upomega \in AS_x(O)$,  then the potential 
\cite[3.1]{R}, \cite[Definition 2]{Kha03}, \cite[3.1, 3.2]{KhaRoz18}, \cite{Chi18}
\begin{equation*}
{\pt}_{\upomega-\delta_x}(y))\overset{\eqref{pot}}{=}
{\pt}_{\upomega}-\pt_{\delta_x}(y)\overset{\eqref{kK}}{=}{\pt}_{\upomega}(y)-K_{d-2}(y, x), \quad y\in \RR^d_{\infty}\!\setminus\! x, 
\end{equation*}
satisfies conditions \cite[\S~1]{Kha03} (see also Duality Theorem \ref{DT3} in Sec.~\ref{DT} below)
\begin{equation}\label{ASpc}
\begin{split}
{\pt}_{\upomega-\delta_x}&\in \sbh(\RR^d_{\infty}\!\setminus\!x),\quad  {\pt}_{\upomega-\delta_x}(\infty):=0,
\\ 
 {\pt}_{\upomega-\delta_x}&\equiv 0 \quad\text{on $\RR^d_{\infty}\!\setminus\!\infill_O (x \cup \supp \upomega )$},\\
{\pt}_{\upomega-\delta_x}(y)&\leq -K_{d-2}(x,y)+O(1)\quad\text{as $x\neq y\to x$}. 
\end{split}
\end{equation}
If $x\in O$ and   $\delta_{x} \preceq_{\sbh(O)} \upomega$, then this measure $\upomega$ is called a {\it Jensen measure on $O$ at $x$}
\cite[3]{Gamelin}, \cite{Koosis}, \cite{Koosis96}, \cite{C-R}, \cite{C-RJ}, \cite{Ransford01}, \cite{Schachermayer}, \cite{HN11}, \cite{HN},  \cite{Kha91}, \cite{KhaTalKha15}, \cite{BaiTalKha16}, \cite{KhaKha19}. The class of these measures is denoted by $J_x(O)$, and properties \eqref{ASpc} are supplemented by the {\it positivity property\/} $\pt_{\upomega-\delta_x}\geq 0$ on $\RR_{\infty}^d\!\setminus\!x$ for all measures $\upomega \in J_x(O)\subset AS_x(O)$. These measures can be considered as generalizations of the extended harmonic measures \eqref{oB}.

By the implication \ref{Im}$\Rightarrow$\ref{Vm} of the Main Lemma  with $\varDelta:=\delta_x$, we obtain the following our  version \cite[Proposition 1.2, (1.3)]{Kha03} of the Poisson-Jensen formula for Arens\,--\,Singer measures $\upomega\in AS_x(O)$, generalizing the classical Poisson-Jensen formula \eqref{clasPJ}.
\begin{PJfASJ}
If  $\upomega \in AS_x(O)$, then 
\begin{equation}\label{PJASJ}
u(x)=\int u\dd \upomega-\int \pt_{\upomega-\delta_x}\dd \varDelta_u
\quad\text{\it for each  $u\in \sbh_*(O)$ with $u(x)\neq -\infty$.} 
\end{equation}
If $\upomega \in J_x(O)$ and $\upomega\neq \delta_x$, then the restriction  $u(x)\neq -\infty$ in \eqref{PJASJ} can be removed.
\end{PJfASJ}
For $x\in \RR^d$, a function $V\in \sbh_*\bigl(\RR^d_{\infty}\!\setminus\! x\bigr)$  
is called an {\it Arens\,--\,Singer potential  on  $O$ with pole at\/} $x\in O$  \cite[3]{Gamelin}, \cite{Sarason}, \cite{Kha01II}, \cite{Kha03}, \cite[Definition 6]{Kha07},  \cite[\S~4]{KudKha09} if 
there is $S_V\Subset O$ such that
\begin{equation}\label{ASpc+}
V\equiv 0 \quad\text{on $\complement S_V$}\quad\text{\it and}\quad 
\limsup_{x\neq y\to x}\frac{V(y)}{-K_{d-2}(x,y)}\leq 1.
\end{equation}
The class of all Arens\,--\,Singer potentials  on $O$ with pole at $x\in O$ denote by $ASP_x(O)$. 
A {\it positive\/} Arens\,--\,Singer potential is called  a {\it Jensen potential  on  $O$ with pole at\/} $x\in O$  \cite[3]{Gamelin},  \cite{Anderson}, \cite{Kha03}, 
\cite{MOS},  \cite[Definition 8]{Kha07}, \cite[IIIC]{Koosis96}, \cite{Kha12}, \cite{KhaTalKha15}, \cite{BaiTalKha16}. We denote by  $JP_x(O)$
the class of all Jensen  potentials  on $O$ with pole at $x\in O$. These potentials can be considered as generalizations of the Green's functions \eqref{gD}.
For  $V\in ASP_x(O)$, we choose (cf. \eqref{pqKo})
\begin{equation}\label{pV}
p\colon y\longmapsto  V(y)+K_{d-2}(y,x), 
\quad   q\colon y\longmapsto K_{d-2}(y,x) 
\quad\text{for  $y\in \RR^d$}.  
\end{equation} 
Then these subharmonic functions on $\RR^d$ 
are harmonic and coincide outside $\clos S_V$ by \eqref{ASpc+}, and the implication [\ref{Is}]$\Rightarrow$[\ref{IIIs}] of the Main Theorem give the equality 
(cf. \eqref{10a})
\begin{equation*}
\int_{O} u\dd \delta_x+ \int_{S_O} 
\Bigl(V(y)+K_{d-2}(\cdot,x)\Bigr)
\dd \varDelta_u =\int_{O} u\dd \varDelta_V
 +\int_{S_O} K_{d-2}(\cdot,x)\dd \varDelta_u, 
\end{equation*}  
where $S_O=\infill \biggl(x\cup \Bigl(\bigcup \bigl\{ y\in O\colon V\notin \har(y)\bigr\}\Bigr)\biggl)$. Hence we obtain 
\begin{PJfASJV}
If\, $V \in ASP_x(O)$, then 
\begin{equation}\label{PJASJV}
u(x)=\int u\dd \varDelta_V-\int V\dd \varDelta_u
\quad\text{\it for each  $u\in \sbh_*(O)$ with $u(x)\neq -\infty$.} 
\end{equation}
If\/ $V \in JP_x(O)$ and\/ $V\not\equiv 0$ on\/ $\complement x$, then the restriction  $u(x)\neq -\infty$ in \eqref{PJASJV} can be removed.
\end{PJfASJV}

\section{Proof of the Main Lemma}\label{PML}
\setcounter{equation}{0}

\subsection{\sf Representations for pairs of subharmonic functions}\label{RSF}

\begin{proposition}\label{prO} If\/ $\mu \in \Meas_{\comp}(\RR^d)$, then
\begin{subequations}\label{pnu0}
\begin{align}
{\pt}_{\mu}&\in \sbh(\RR^d)\bigcap \har(\RR^d\!\setminus\!\supp \mu), 
\tag{\ref{pnu0}h}\label{{pnu0}h}
\\ 
{\pt}_{\mu}(x)&\overset{ \eqref{{kK}k}}{=}\mu (\RR^d)k_{d-2}\bigl(|x|\bigr)+O\bigl(1/|x|^{d-1}\bigr),  
\quad x\to \infty.
\tag{\ref{pnu0}$\infty$}\label{{pnu0}infty}
\end{align}
\end{subequations}
\end{proposition}

\begin{proof} For $d=1$, we have
\begin{equation*}
\left|\pt_{\mu}(x)-\mu(\RR)|x|\right|\leq \int \bigl||x-y|-|x|\bigr|\dd |\mu|(y)\leq \int |y|\dd |\mu|(y)
=O(1), \quad |x|\to +\infty. 
\end{equation*} 
 
See \cite[Theorem 3.1.2]{R} for $d=2$. 

For $d>2$ and $|x|\geq 2\sup\bigl\{|y|\colon y\in \supp \mu\bigr\}$, we have 
\begin{multline*}
\left|{\pt}_{\mu}(x)-\mu (\RR^d)k_{d-2}\bigl(|x|\bigr)\right|
=\left|\int \left(\frac{1}{|x|^{d-2}}-\frac{1}{|x-y|^{d-2}}\right)
\dd \mu (y)\right|\\
\leq\int \left|\frac{1}{|x|^{d-2}}-\frac{1}{|x-y|^{d-2}}\right|
\dd |\mu|(y)
\leq \frac{2^{d-2}}{|x|^{2d-4}}\int \left||x-y|^{d-2}
-|x|^{d-2}\right| \dd |\mu|(y)
\\
\leq \frac{2^{d-2}}{|x|^{2d-4}}
\int |y||x|^{d-3} \sum_{k=0}^{d-3}\Bigl(\frac32\Bigr)^k 
\dd |\mu|(y)\leq 2\frac{3^{d-2}}{|x|^{d-1}}
\int |y|\dd |\mu|(y)=O\Bigl(\frac{1}{|x|^{d-1}}\Bigr).
\end{multline*}
\end{proof}

\begin{theorem}\label{lemPQ} Let $O\subset \RR^d$ be an open set, and let  $p\in \sbh_*(O)$ and $q\in \sbh_*(O)$ be pair of functions such that $p$ and $q$ are harmonic outside a compact subset in $O$. If there is a compact set $S\Subset O$ such that $p=q$ on $O\!\setminus\!S$, then, for Riesz measures $\varDelta_p\in \Meas_{\comp}^+(O)$ of $p$ and  $\varDelta_q\in \Meas_{\comp}^+(O)$ of $q$, we have 
\begin{equation}\label{ptd}
\varDelta_p(O)=\varDelta_q(O), 
\quad \text{$\pt_{\varDelta_p}=\pt_{\varDelta_q}$ on $\RR^d\!\setminus\!S$},
\end{equation}
and there is a harmonic function $H$ on $O$  such that 
\begin{equation}\label{PQpR}
\begin{cases}
p=\pt_{\varDelta_p}+H\\
q=\pt_{\varDelta_q}+H
\end{cases}
\quad \text{on $O$, \quad  $H\in \har(O)$}.
\end{equation}
\end{theorem}
\begin{proof} By Weyl's lemma on the Laplace equation, we have  
\begin{equation*}
\begin{cases}
\bigtriangleup(p-\pt_{\varDelta_p})\overset{\eqref{df:cm}}{=}\frac{1}{c_d}(\varDelta_p-\varDelta_p)=0\\
\bigtriangleup(q-\pt_{\varDelta_q})\overset{\eqref{df:cm}}{=}\frac{1}{c_d}(\varDelta_q-\varDelta_q)=0
\end{cases}\quad \Longrightarrow \quad 
\begin{cases}
h_p:=p-\pt_{\varDelta_p}\in \har(O)\\
h_q:=q-\pt_{\varDelta_q}\in \har(O)
\end{cases}
\end{equation*}
 and  obtain representations
\begin{equation}\label{hPQ}
\begin{cases}
p=\pt_{\varDelta_p}+h_p\\
q=\pt_{\varDelta_Q}+h_q
\end{cases}
\text{on $O$  with $h_p\in \har(O)$ and $h_q\in \har(O)$}.
\end{equation}

Let us first consider separately 

\paragraph{\bf The case $O:=\RR^d$ in the notation $P: = p$ and $Q: = q$.} 
Put 
\begin{equation}\label{hpq}
h\overset{\eqref{hPQ}}{:=}h_P-h_Q\in \har(\RR^d).
\end{equation}
By the conditions of Theorem \ref{lemPQ} and Proposition \ref{prO}, we have
\begin{multline}\label{bd}
h(x)\overset{\eqref{hpq}}{=}h_P(x)-h_Q(x)\overset{\eqref{hPQ}}{=}-\pt_{\varDelta_P}(x)+\pt_{\varDelta_Q}(x)+\bigl(P(x)-Q(x)\bigr)\\
\overset{\eqref{{pnu0}infty}}{=}
bk_{d-2}\bigl(|x|\bigr)
+O\bigl(|x|^{1-d}\bigr), \quad |x|\to +\infty, \quad\text{where  $b:=\varDelta_Q(\RR^d)-\varDelta_P(\RR^d)$.}
\end{multline}

\paragraph{The case $d>2$.} If $d\geq 3$, then, in view of \eqref{bd}, this harmonic function $h$ bounded 
on $\RR^d$. By Liouville's Theorem \cite[Ch.~3]{ABR}, $h$ is constant, and 
$h_P-h_Q=h\overset{\eqref{bd}}{\equiv} 0$ on $\RR^d$. 
In particular, $|b|=\bigl|b+|x|^{d-2}h(x)\bigr|\overset{\eqref{bd}}{=}
O\bigl(1/|x|\bigr) $ as $x\to \infty$, i.e., $b=0$. 
Thus, for   $H:=h_P=h_Q$,   
 by \eqref{hPQ}, we obtain representations \eqref{PQpR} together with  $\pt_{\varDelta_P}=\pt_{\varDelta_Q}$ on $\RR^d\!\setminus\!S$,  as required. 
  
\paragraph{The case $d=2$.} Using  \eqref{bd} we obtain
$\bigl|h(x)-b\log |x|\bigr|\overset{\eqref{bd}}{=}O\bigl(1/|x|\bigr)$ as $x\to \infty$. Hence, this harmonic function $h$ is bounded from below
 if $b\geq 0$ or  bounded from above if $b<0$. Therefore, by  Liouville's Theorem,  $h$ is constant, $b=0$, i.e., $\varDelta_P(\RR^2)\overset{\eqref{bd}}{=}\varDelta_Q(\RR^2)$, and $h
\overset{\eqref{bd}}{\equiv} 0$ on $\RR^2$. Thus, we obtain  \eqref{PQpR} together with \eqref{ptd}.  
 \paragraph{The case $d=1$.} Using  \eqref{bd} we obtain
$\bigl|h(x)-b|x|\bigr|\overset{\eqref{bd}}{=}O(1)$ as $x\to \infty$. Hence, this affine  function $h$
on $\RR$ is bounded from below  if $b\geq 0$ or bounded from above if $b<0$. 
Therefore, $h$ is constant, $b=0$, i.e., $\varDelta_P(\RR)\overset{\eqref{bd}}{=}\varDelta_Q(\RR)$, and 
$h\overset{\eqref{bd}}{\equiv} C$  on $\RR$ for a constant $C\in \RR$.  Thus, 
 \begin{equation}\label{hPQc}
\begin{cases}
P(x)=\pt_{\varDelta_P}(x)+ax+b+C\\
Q(x)=\pt_{\varDelta_Q}(x)+ax+b
\end{cases}
\text{for  $x\in \RR$  with $h_Q(x)\underset{\text{\tiny $x\in \RR$}}{\equiv} ax+b$},
\end{equation}  
The definition \eqref{pot}  of potentials in the case $d=1$  immediately implies 
\begin{lemma}\label{lem2} Let $\varDelta\in \Meas^+_{\comp}(\RR)$, and
$s_l:=\inf \supp \varDelta$, $s_r:=\sup \supp \varDelta$. 
Then 
\begin{equation*}
\pt_\varDelta(x)=
\begin{cases}
\varDelta(\RR)x-\int y\dd \varDelta(y)&\text{if $x\geq s_r$},\\
-\varDelta(\RR)x+\int y\dd \varDelta(y)&\text{if $x\leq s_l$}.
\end{cases}
\end{equation*}
\end{lemma}
We set \begin{equation*}
\begin{cases}
t:=\varDelta_P(\RR)=\varDelta_Q(\RR)\in \RR^+,\\ 
S_l:=\inf (S\cup \supp \varDelta_P\cup \supp \varDelta_Q)\in \RR,\\
S_r:=\sup (S\cup \supp \varDelta_P\cup \supp \varDelta_Q)\geq S_l.
\end{cases}
\end{equation*}
In view of $P(x)\equiv Q(x)$ for $x\in \RR\!\setminus\!S$,   by Lemma  \ref{lem2}, we have    
\begin{equation*}
\begin{cases}
tx-\int y\dd \varDelta_P(y)+ax+b+C=
tx-\int y\dd \varDelta_Q(y)+ax+b \quad\text{if $x\geq S_r$},\\
-tx+\int y\dd \varDelta_P(y)+ax+b+C=
-tx+\int y\dd \varDelta_Q(y)+ax+b \quad\text{if $x\leq S_l$},
\end{cases}
\end{equation*}
whence
\begin{equation*}
\begin{cases}
-\int y\dd \varDelta_P(y)+C=
-\int y\dd \varDelta_Q(y),\\
\int y\dd \varDelta_P(y)+C=
\int y\dd \varDelta_Q(y).
\end{cases}
\end{equation*}
Adding these equalities, we obtain $C=0$. Thus, we get 
 \eqref{PQpR} together with \eqref{ptd}.

\paragraph{\bf The general case of an open set  $O\subset \RR^d$.}
Let's start again with the representations \eqref{hPQ}. We set  
 \begin{subequations}\label{d}
\begin{align}
{\mathsf S}&\overset{\text{\tiny closed}}{:=}S\bigcup \supp \varDelta_q\bigcup \supp \varDelta_p\overset{\text{\tiny compact}}{\Subset} O,
\tag{\ref{d}S}\label{{d}S}\\
w:=p-q,&\quad  \varDelta_w\overset{\eqref{df:cm}}{:=}c_d\bigtriangleup\!w=\varDelta_p-\varDelta_q\in \Meas({\mathsf S})\subset
\Meas_{\comp}(O).
\tag{\ref{d}w}\label{{d}w}
\end{align}
\end{subequations}
This difference $w\in \sbh_*(O)-\sbh_*(O)$ of subharmonic functions, i.e., a $\delta$-subharmonic function   \cite{Arsove}, \cite{Arsove53p}, \cite[3.1]{KhaRoz18},  is uniquely defined on $O$ outside a polar set (cf. \eqref{Dom})
\begin{equation}\label{Domd}
\Dom w 
:=\left\{x\in O\colon
\inf\left\{ \int_{0}\frac{\varDelta_w^-(x,t)}{t^{d-1}} \dd t, \int_{0}\frac{\varDelta_w^+(x,t)}{t^{d-1}} \dd t\right\}<+ \infty 
\right\}\overset{\eqref{{d}S}}{\subset} {\mathsf S},
\end{equation}
and $w\equiv 0$ on $O\!\setminus\!{\mathsf S}$ since $p=q$ outside $S\subset {\mathsf S}$ in \eqref{{d}w}, and $p,q\in \har(O\!\setminus\!{\mathsf S})$.  The Riesz charge $\varDelta_w
\overset{\eqref{d}}{\in} \Meas_{\comp}(O)$ of this $\delta$-subharmonic function $w$ on $O$ is also uniquely determined on $O$ with $\supp |\varDelta_w|\subset {\mathsf S}$ \cite[Theorem 2]{Arsove}. 
The function $w\colon O\!\setminus\!\Dom w\to \overline \RR$ can be extended from $O$ to the whole of $\RR^d\!\setminus\!\Dom w$ 
 by zero values:
\begin{equation}\label{d0}
w\equiv 0 \quad\text{on }\RR^d\!\setminus\!{\mathsf S}\overset{\eqref{{d}S}}{\supset} \RR^d\!\setminus\!O, 
\quad \varDelta_w=\varDelta_p-\varDelta_q\overset{\eqref{{d}w}}{\in} \Meas({\mathsf S}). 
\end{equation}
This function $w$ on $\RR^d\!\setminus\!\Dom w$
 is still a $\delta$-subharmonic function, but already on $\RR^d$, since $\delta$-subharmonic functions are defined locally \cite[Theorem 3]{Arsove}. The 
Riesz charge of this  $\delta$-subharmonic function $w\colon \RR^d\!\setminus\!\Dom d \to \overline \RR$ on $\RR^d$ is the same charge
$\varDelta_d\overset{\eqref{{d}w}}{\in} \Meas({\mathsf S})$. There is a canonical representation  \cite[Definition 5]{Arsove} of $w$ such that \cite[Theorem 5]{Arsove}
\begin{subequations}\label{PQ}
\begin{align}
w=P-Q \quad\text{on }\RR^d\!\setminus\!\Dom w, &\quad\text{where }
P,Q\in \sbh_*(\RR^d)\cap\har(\RR^d\!\setminus\!{\mathsf S})
\tag{\ref{PQ}d}\label{{PQ}s}
\\
\intertext{are functions with Riesz measures} 
\varDelta_P\overset{\eqref{df:cm}}{:=}c_d\bigtriangleup\!P=\varDelta_w^+
\overset{\eqref{{PQ}s}}{\in} \Meas^+({\sf S}),& \quad 
\varDelta_Q\overset{\eqref{df:cm}}{:=}c_d\bigtriangleup\!Q=\varDelta_w^-\overset{\eqref{{PQ}s}}{\in} \Meas^+({\sf S}),
\tag{\ref{PQ}$\varDelta$}\label{{PQ}D}
\\
P\overset{\eqref{d0},\eqref{{PQ}s}}{\equiv}Q \quad \text{on }\RR^d\!\setminus\!{\mathsf S},&
\tag{\ref{PQ}$\equiv$}\label{{PQ}d}
\\
\intertext{and there is a function $s\in \sbh_*(O)$ with Riesz measure}
\varDelta_s=\varDelta_p-\varDelta_w^+\overset{\eqref{d0},\eqref{{PQ}D}}{=}&\varDelta_q-\varDelta_w^- \in \Meas^+({\mathsf S})
\tag{\ref{PQ}s}\label{{PQ}vD}
\\
\quad\text{such that } &\begin{cases}
p=P+s,\\
q=Q+s
\end{cases}
\quad \text{on } O.
\tag{\ref{PQ}r}\label{{PQ}r}
\end{align}
\end{subequations}
By \eqref{{PQ}s} and \eqref{{PQ}d},  all conditions  of Theorem \ref{lemPQ} 
are fulfilled for functions $P,Q$ from \eqref{PQ} instead of $p,q$, but in the case  $\RR^d$ instead of $O$ and ${\mathsf S}$ instead of $S$.  Thus,  we have 
\eqref{ptd} in the form 
\begin{subequations}\label{dpt}
\begin{align}
\varDelta_w^+(O)\overset{\eqref{{PQ}D}}{=}\varDelta_P(\RR^d)
\overset{\eqref{ptd}}{=}\varDelta_Q(\RR^d)
\overset{\eqref{{PQ}D}}{=}\varDelta_w^-(O), 
\tag{\ref{dpt}$\varDelta$}\label{{dpt}d}
\\
\pt_{\varDelta_w^+}\overset{\eqref{{PQ}D}}{=}\pt_{\varDelta_P}=\pt_{\varDelta_Q}\overset{\eqref{{PQ}D}}{=}\pt_{\varDelta_w^-} \quad \text{on $\RR^d\!\setminus\!{\mathsf S}$},
\tag{\ref{dpt}p}\label{{dpt}pt}
\end{align}
\end{subequations}
and  the representations \eqref{PQpR} in the form
\begin{equation}\label{PQpR+}
\begin{cases}
P\overset{\eqref{PQpR}}{=}\pt_{\varDelta_P}+h
\overset{\eqref{{dpt}pt}}{=}\pt_{\varDelta_w^+}+h\\
Q\overset{\eqref{PQpR}}{=}\pt_{\varDelta_Q}+h
\overset{\eqref{{dpt}pt}}{=}\pt_{\varDelta_w^-}+h
\end{cases}
\quad \text{on $\RR^d$, \quad  $h\in \har(\RR^d)$}.
\end{equation}
Hence, by representation  \eqref{{PQ}r},  we obtain the following representations
\begin{equation}\label{rPQ}
\begin{split}
&\begin{cases}
p\overset{\eqref{{PQ}r},\eqref{PQpR+}}{=}\pt_{\varDelta_w^+}+h+s,\\
q\overset{\eqref{{PQ}r},\eqref{PQpR+}}{=}\pt_{\varDelta_w^-}+h+s
\end{cases}
\quad \text{on  $O$}, \\
h\in \har(\RR^d),\quad &\pt_{\varDelta_w^+}\overset{\eqref{{dpt}pt}}{=}\pt_{\varDelta_w^-}
\text{ on $\RR^d\!\setminus\!{\mathsf S}$,}\quad
\varDelta_w^+(O)\overset{\eqref{{dpt}d}}{=}\varDelta_w^-(O). 
 \end{split}
\end{equation} 
Besides,  the function  $l\overset{\eqref{{PQ}vD}}{:=}s-\pt_{\varDelta_s}$ 
is harmonic on $O$ by Weyl's lemma on the Laplace equation 
$\bigtriangleup\,(s-\pt_{\varDelta_s})\overset{\eqref{{PQ}vD}}{=}\varDelta_s-\varDelta_s=0$.  Hence 
\begin{equation}\label{rPQs}
\begin{split}
&\begin{cases}
p\overset{\eqref{rPQ}}{=}\pt_{\varDelta_w^+}+\pt_{\varDelta_s}+h+l,\\
q\overset{\eqref{rPQ}}{=}\pt_{\varDelta_w^-}+\pt_{\varDelta_s}+h+l
\end{cases}
\quad \text{on  $O$, where $h\in \har(\RR^d)$ and  $l\in \har(O)$,} 
\\
& \pt_{\varDelta_w^+}+\pt_{\varDelta_s}\overset{\eqref{rPQ}}{=}\pt_{\varDelta_w^-}+\pt_{\varDelta_s}\text{ on } \RR^d\!\setminus\!{\mathsf S},\quad 
\varDelta_w^+(O)\overset{\eqref{rPQ}}{=}\varDelta_w^-(O). 
\end{split}
\end{equation}
By construction, we have
\begin{equation*}
\begin{cases}
\pt_{\varDelta_w^+}+\pt_{\varDelta_s}=\pt_{\varDelta_w^++\varDelta_s}
\overset{\eqref{{PQ}vD}}{=}\pt_{\varDelta_p},
\\
\pt_{\varDelta_w^-}+\pt_{\varDelta_s}=\pt_{\varDelta_w^-+\varDelta_s}
\overset{\eqref{{PQ}vD}}{=}\pt_{\varDelta_q},
\end{cases} 
\varDelta_p(O)=(\varDelta_w^++\varDelta_s)(O)
\overset{\eqref{{PQ}vD}}{=}
(\varDelta_w^-+\varDelta_s)(O)=\varDelta_p(O).
\end{equation*}
Hence, if we set $H:=h+l\in \har(O)$, then, by \eqref{rPQs}, we obtain 
exactly \eqref{PQpR}, as well as \eqref{ptd}, with the only difference being that in \eqref{ptd} we have ${\mathsf S}\overset{\eqref{{d}S}}{\supset} S$ instead of $S$. 
Moreover, it immediately follows from the representation \eqref{PQpR} and the condition $p = q$ on ${\mathsf S}\!\setminus\!S\overset{\eqref{{d}S}}{\subset} O\!\setminus\!S$ that $\pt_{\varDelta_p}=\pt_{\varDelta_q}$ on $\RR^d\!\setminus\!S=
(\RR^d\!\setminus\!{\mathsf S})\bigcup ({\mathsf S}\!\setminus\!S)$. Theorem \ref{lemPQ} is proved.
\end{proof} 

\subsection{\sf Duality between balayage of measures and their potentials}\label{Db}

In this Subsec.~\ref{Db}, the equivalence of the first four statements of the Main Lemma according to the scheme 
\begin{equation}\label{diag}
\begin{array}{ccc}
\text{\ref{Im}} &\longrightarrow & \text{\ref{IIm}} \\
 & \nwarrow &  \downarrow\\
\text{\ref{IVm}} &  \longleftrightarrow& \text{\ref{IIIm}}
\end{array}
\end{equation}
will be established. 
We write A$\overset{{\it \tiny proof}}{\Longrightarrow}$B
if  the implication A$\Rightarrow$B is proved or discussed below.

\ref{Im}$\overset{{\it \tiny proof}}{\Longrightarrow}$\ref{IIm}.
By Proposition \ref{KOc}(\ref{Ki}-\ref{Kii}) and \cite[Theorem 1.7]{Gardiner}, if $h\in \har(S_O)$ is harmonic  on  the inward   filling
$S_O\overset{\eqref{SO}}{=}\infill(\supp\varDelta\cup \supp \upomega)=\infill S_O\Subset O$
 of $S$,   then there are functions $h_k\underset{\text{\tiny $k\in \NN$}}{\in} \har (O)$ such that the sequence  $(h_k)_{k\in \NN}$ converges to $h$ in the space  $C (S_O)$
of all continuous functions on the compact set $S_O\Subset O$ with $\sup$-norm. Hence, 
\begin{multline*}
\int_{S_O} h \dd {\varDelta}= \int_{S_O}  \lim_{k\to \infty}
h_k \dd {\varDelta}  = \lim_{k\to \infty} \int_{S_O} h_k \dd {\varDelta}= \lim_{k\to \infty} \int_{O} h_k \dd {\varDelta}
\\\overset{\text{\ref{Im}},\eqref{bhar}}{=}\lim_{k\to \infty} \int_{O} h_k \dd \upomega= \lim_{k\to \infty} \int_{S_O} h_k \dd \upomega= 
 \int_{S_O} \lim_{k\to \infty} h_k \dd \upomega=
 \int_{S_O} h \dd {\varDelta}.
\end{multline*}
The statement \ref{IIm} of the Main Lemma  is established.

\ref{IIm}$\overset{{\it \tiny proof}}{\Longrightarrow}$\ref{IIIm}.
If $x\notin S_O$, then the subharmonic function 
\begin{equation}\label{ux}
u_x\colon y\underset{\text{\tiny $y\in \RR^d$}}{\longmapsto} K_{d-2}(y,x)
\end{equation}
is harmonic on $S_O$. Hence, for $x\notin S_O$, 
\begin{multline*}
\pt_{\varDelta}(x)=\int_{\supp \varDelta}K_{d-2}(\cdot,x)\dd \varDelta=
\int_{S_O}K_{d-2}(\cdot,x)\dd \varDelta
\overset{\eqref{ux}}{=}\int_{S_O}u_x\dd \varDelta
\\
\overset{\text{\ref{IIm}},\eqref{bhar}}{=}
\int_{S_O}u_x\dd \upomega
\overset{\eqref{ux}}{=}\int_{S_O}K_{d-2}(\cdot,x)\dd \upomega=\int_{\supp \upomega}K_{d-2}(\cdot,x)\dd \varDelta=\pt_{\upomega}(x).
\end{multline*}
The statement \ref{IIIm} of the Main Lemma is established.

\ref{IIIm}$\overset{{\it \tiny proof}}{\Longrightarrow}$\ref{IVm}.
This implication is obvious if we choose $p:=\pt_{\upomega}$ and $q:=\pt_{\varDelta}$.

\ref{IVm}$\overset{{\it \tiny proof}}{\Longrightarrow}$\ref{IIIm}.
This implication is a special case of Theorem \ref{lemPQ} with the conclusion \eqref{ptd}.

\ref{IIIm}$\overset{{\it \tiny proof}}{\Longrightarrow}$\ref{Im}.
We use the following 
\begin{lemma}[{\rm \cite[Lemma 1.8]{Gardiner}}]\label{l1} Let $F$ be a compact subset in\/  $\RR^d$, $h\in \har (F)$, and $b \in \RR^+\!\setminus\!0$. Then  there are points $y_1,y_2,\dots, y_m$ in $\RR^d\!\setminus\! F$ and constants  $a_1, a_2, \dots , a_m\in \RR$ such that
\begin{equation}\label{hk}
\Bigl|h(x)-\sum_{j=1}^{m} a_j k_{d-2}\bigl(|x-y_j|\bigr)\Bigr|<b
\quad\text{for all $x\in F$}.
\end{equation}
\end{lemma}
Applying Lemma \ref{l1} to the compact set 
$F:=S_O \Subset O$
and a function $h\in \har (O)$, we obtain
\begin{multline*}
\Bigl|\int_O h \dd  (\upomega -\varDelta)\Bigl|
=\Bigl|\int_{S_O} h \dd  (\upomega -\varDelta)\Bigl|
\overset{\text{\ref{IIIm}}}{=}
\biggl|\int_{S_O} h \dd (\upomega -\varDelta)-\sum_{j=1}^{m}a_j \Bigl(\overset{0}{\overbrace{{\pt}_{\upomega} (y_j)- {\pt}_{\varDelta}(y_j)}}\Bigr) \biggr|
\\
\overset{\eqref{pot}}{=}
\biggl|\int_{S_O} h \dd (\upomega -\varDelta)-\sum_{j=1}^{m}a_j \Bigl(\int_{S_O}K_{d-2}(y,y_j) \dd\upomega (y)- 
\int_{S_O}K_{d-2}(y,y_j) \dd \varDelta(y)\Bigr) \biggr|
\\
\overset{\eqref{{kK}K}}{=}
\biggl|\int_{S_O} h (y) \dd (\upomega -\varDelta)(y)-
\int_{S_O}\sum_{j=1}^{m}a_j k_{d-2}\bigl(|y-y_j|\bigr) \dd (\upomega-\varDelta)(y) \biggr|
\\
\overset{\eqref{hk}}{\leq} \sup_{y\in S_O} \Bigl| h (y)-\sum_{j=1}^{m} a_jk_{d-2}\bigl(|y-y_j|\bigr)\Bigr| \bigl(\upomega(O)+\varDelta (O)\bigr)
\overset{\eqref{hk}}{\leq}  b \bigl(\upomega(O)+\varDelta (O)\bigr)
\end{multline*}
for each  $b\in \RR^+\!\setminus\!0$. Hence $\varDelta\preceq_{\har(O)}\upomega$. Thus,
we  obtain \ref{Im} and   complete \eqref{diag}.  

\subsection{\sf The symmetric Poisson\,--\,Jensen formula for measures and their potentials}\label{PJsf} 

In this Subsec.  \ref{PJsf},  we complete the proof of the Main Lemma by the scheme
\begin{equation}\label{diag2}
\bigl(\text{\ref{IIm}}\cap \text{\ref{IIIm}}\bigr) \longrightarrow  \text{\ref{Vm}}
 \longrightarrow \text{\ref{VIm}} \longrightarrow \text{\ref{VIIm}}
 \longrightarrow \text{\ref{IVm}} ,
\end{equation}
where $\text{\ref{IIm}}\cap \text{\ref{IIIm}}$ means that statements \ref{IIm} and \ref{IIIm} are simultaneously satisfied, and  the equivalence $(\text{\ref{IIm}}\cap \text{\ref{IIIm}})\Leftrightarrow$\ref{IVm} of the extreme statements $(\text{\ref{IIm}}\cap \text{\ref{IIIm}})$  and \ref{IVm} of \eqref{diag2} has already been proved in the previous Subsec.~\ref{Db} by the scheme \eqref{diag}.   

$(\text{\ref{IIm}}\cap \text{\ref{IIIm}})\overset{{\it \tiny proof}}{\Longrightarrow}$\ref{Vm}. 
Let $u\overset{\eqref{{fPJ}B}}{\in} \sbh_*(\clos B )$, where $S_O\overset{\eqref{{fPJ}B}}{\subset} B\Subset O$.
We can choose an  open set $O'$ such that $B\Subset O'\Subset O$ and 
$u\in \sbh_*(\clos O')$. Consider first the case 
\begin{equation}\label{tfin}
-\infty<\int u\dd \varDelta,\quad \text{where $\supp \varDelta\overset{\eqref{SO}}{\subset} S_O\Subset O'$.}
\end{equation}
Let 
\begin{equation}\label{mu'}
\mu':=\varDelta_u\bigm|_{\clos O'}
\end{equation} 
be the restriction of Riesz measure of $u\in \sbh_*(\clos O')$ 
to $\clos O'\Subset O$.  By the Riesz Decomposition Theorem   \cite[Theorem 3.7.1]{R}, \cite[Theorem 3.9]{HK}, \cite[Theorem 4.4.1]{AG}, \cite[Theorem 6.18]{Helms} we obtain a representation  
\begin{equation}\label{RDT}
u={\pt}_{\mu'}+h\quad \text{on } O',\quad \text{where  $h\in \har (O' )$ is continuous and bounded on $S_O$.}
\end{equation}
Integrating this representation with respect to $\dd \upomega$ and $\dd \varDelta$, we obtain 
\begin{subequations}\label{utm}
\begin{align}
\int u \dd \upomega&\overset{\eqref{RDT}}{=} \int {\pt}_{\mu'} \dd \upomega +\int h \dd \upomega, \quad \supp \upomega\overset{\eqref{SO}}{\subset} S_O, 
\tag{\ref{utm}$\upomega$}\label{{utm}m}
\\
\int u \dd \varDelta &\overset{\eqref{RDT}}{=} \int {\pt}_{\mu'} \dd \varDelta  +\int h \dd \varDelta, \quad \supp \varDelta \overset{\eqref{SO}}{\subset} S_O,
\tag{\ref{utm}$\varDelta$}\label{{utm}o}
\end{align}
\end{subequations}
where the three integrals in \eqref{{utm}o}  are finite, although in the equality \eqref{{utm}m} the first two integrals can take simultaneously the value of $-\infty$, but the last integral in \eqref{{utm}m} is finite.  Therefore, the difference \eqref{{utm}m}$-$\eqref{{utm}o} of these two equalities is well defined:
\begin{equation}\label{if}
\int u \dd \upomega-\int u \dd \varDelta\overset{\eqref{utm}}{=} \int {\pt}_{\mu'}\dd \upomega-
 \int {\pt}_{\mu'} \dd \varDelta  +\int_{S_O} h \dd (\upomega-\varDelta),
\end{equation}
 where the first and third integrals can simultaneously take the value of $-\infty$, and the remaining integrals are finite. By the statement \ref{IIm}
we have $\varDelta\simeq_{\har(S_O)}\upomega$. Hence the\textit{ last integral\/} in \eqref{if} \textit{vanishes\/} according to  \eqref{bhar}.   
Using Fubini's Theorem on repeated integrals, in view of the symmetry property of kernel  $K_{d-2}$ in \eqref{{kK}K},  we have
\begin{multline}\label{ct}
\int {\pt}_{\mu'}\dd \varDelta=\int \int 
K_{d-2}(y,x) \dd \mu' (y)\dd \varDelta (x)\\
=\int  \int K_{d-2}(x,y) \dd \varDelta (x) \dd \mu' (y)\overset{\eqref{mu'}}{=}\int_{\clos O'}{\pt}_{\varDelta} \dd \varDelta_u,
\end{multline}
and the same way
\begin{multline}\label{chmu}
\int {\pt}_{\mu'}\dd \upomega =\int \int K_{d-2}(y,x) \dd \mu' (y)\dd \upomega (x)\\ =\int  \int K_{d-2}(x,y) \dd \upomega (x) \dd \mu' (y)\overset{\eqref{mu'}}{=}\int_{\clos O'}{\pt}_{\upomega} \dd \varDelta_u
\end{multline}
even if the integral on the left side of equalities \eqref{chmu} takes the value  $-\infty$ because the integrand $K_{d-2}(\cdot, \cdot )$ is bounded from above on the compact set  $\clos O'\times \clos O'$ \cite[Theorem 3.5]{HK}. Hence  equality  \eqref{if} can be rewritten as
\begin{equation*}
\int u \dd \upomega-\int u \dd \varDelta= \int_{\clos O'}{\pt}_{\upomega} \dd \varDelta_u-  \int_{\clos O'}{\pt}_{\varDelta} \dd \varDelta_u 
\end{equation*}
or in  the form
\begin{equation}\label{OB}
\int u \dd \upomega+\int_{\clos O'}{\pt}_{\varDelta} \dd \varDelta_u
= \int u \dd \varDelta+\int_{\clos O'}{\pt}_{\upomega} \dd \varDelta_u. 
\end{equation}
But by the statement \ref{IIIm}, we have 
\begin{equation*}
{\pt}_{\upomega}\overset{\rm \ref{IIIm}}{=}{\pt}_{\varDelta} \quad \text{on $\RR^d\!\setminus\! S_O\supset \!\clos O'\!\setminus\! B$}. 
\end{equation*} 
Hence, by equality \eqref{OB}, we obtain equality \eqref{{fPJ}f}
 in the case \eqref{tfin}. 

If condition \eqref{tfin}  is not fulfilled, then from the representation \eqref{{utm}o} it follows that the integral on the left-hand side of \eqref{ct} also takes the value $-\infty$. The equalities \eqref{ct} is still true  \cite[Theorem 3.5]{HK}. Hence, the second  integral on the right side of the formula \eqref{{fPJ}f} also takes the value $-\infty$ and this formula \eqref{{fPJ}f} remains true. 

\ref{Vm}$\overset{{\it \tiny proof}}{\Longrightarrow}$\ref{VIm}.
Let   $q\in \sbh_*(S_O)$ be a function with  Riesz measure $\varDelta_q=\varDelta$. Then there is  a function $h\in \har(O)$ such 
that $q=\pt_{\varDelta}+h$ on $O$. By the statement  \ref{Vm}, we have 
\eqref{{fPJ}f} for $B=S_O$. If we set $p:=\pt_{\upomega}+h$, then $\varDelta_p=\upomega$, and  \eqref{2R} follows from   \eqref{{fPJ}f} with $B=S_O$.

\ref{VIm}$\overset{{\it \tiny proof}}{\Longrightarrow}$\ref{VIIm}.
We set $q:=\pt_{\varDelta}\in \sbh_*(\RR^d)$ with $\varDelta_q=\varDelta$. By statement \ref{VIm}, there is a function   $p\in \sbh_*(\RR^d)$ with $\varDelta_p=\upomega$ such that we have  \eqref{2R}. In particular, the equality in \eqref{2R} is true for each special subharmonic function $u_x\colon y\underset{\text{\tiny $y\in \RR^d$}}{\longmapsto} K_{d-2}(y,x)$,  $x\in \RR^d$, and we obtain \ref{VIIm}.

\ref{VIIm}$\overset{{\it \tiny proof}}{\Longrightarrow}$\ref{IVm}.
Each special function $u_x$ in \ref{VIIm} is subharmonic on $\RR^d$ with Riesz measure  $\delta_x$. If $x\in O\!\setminus\!S$, where $S_O\subset S\Subset O$, then $S_O\cap \supp \delta_x=\varnothing$. Thus,   
\begin{equation}\label{2R+}
\int_{S_O} p\dd \delta_x =\int_{S_O}q\dd \delta_x=0
\quad\text{for each $x\in O\!\setminus\!S$}. 
\end{equation}
Hence, by  \eqref{2R} with $u_x$ instead of $u$, we obtain 
\begin{equation*}
\pt_{\varDelta}(x)=\int_{S_O} K(y,x) \dd \varDelta(y)=\int_{S_O} u_x \dd \varDelta\overset{\eqref{2R},\eqref{2R+}}{=}
\int_{S_O} u_x \dd \upomega=\int_{S_O} K(y,x) \dd \upomega(y)=\pt_{\upomega}(x)
\end{equation*} 
for each  $x\in O\!\setminus\!S$. Thus, we obtain the statement \ref{IVm}
for $q:=\pt_{\varDelta}$ and $p:=\pt_{\upomega}$.

The Main Lemma is proved.

\section{Proof of the Main Theorem}\label{PMT} 
\setcounter{equation}{0}

[\ref{Is}]$\overset{{\it \tiny proof}}{\Longrightarrow}$[\ref{IIs}]. 
Without loss of generality, we can assume that $S=S_O$ in \eqref{{Spq}S}. 
Then   the statement [\ref{IIs}] with \eqref{pqh} follows from
Theorem \ref{lemPQ} with \eqref{ptd}--\eqref{PQpR}. 

[\ref{IIs}]$\overset{{\it \tiny proof}}{\Longrightarrow}$[\ref{IIIs}].
By the equality  $\pt_{\varDelta_q}\overset{\eqref{pqh}}{=}\pt_{\varDelta_p}$ on $\RR^d\!\setminus\!S_O$, we have the statement  \ref{IIIm} of the Main  Lemma
for $\varDelta:=\varDelta_q\in \Meas^+(S)$ and $\upomega:=\varDelta_p\in \Meas^+(S)$. 
By implication  \ref{IIIm}$\Rightarrow$\ref{Vm} of the Main Lemma, we obtain 
\begin{equation}\label{fPJpq}
\int_{\supp \varDelta_q} u\dd \varDelta_q+\int_B {\pt}_{\varDelta_p} \dd \varDelta_u\overset{\eqref{fPJ}}{=}
 \int_{\supp \varDelta_p} u\dd \varDelta_p +\int_B {\pt}_{\varDelta_q} \dd \varDelta_u
\end{equation}
for each $B\in \Borel(\RR^d)$ under $S_O\subset B\Subset O$
and  for each $u\in \sbh_*(\clos B)$, where we returned to the separate notation $S\subset S_O:=\infill S$. Obviously, 
 \begin{equation}\label{h}
\int_B h\dd \varDelta_u=\int_B h\dd \varDelta_u
\quad\text{for each $u\in \sbh_*(\clos B)$ and $h\in \har(O)$}.
\end{equation}
Adding \eqref{fPJpq} and \eqref{h}, according to representations \eqref{pqh} of $q$ and $p$, we obtain  
\begin{equation}\label{fPJpq+}
\int_S u\dd \varDelta_q+\int_B p \dd \varDelta_u
\overset{\eqref{fPJ}}{=}
 \int_S u\dd \varDelta_p +\int_B q \dd \varDelta_u,
\end{equation}
where $B$   can be replaced with $B\cap S_{\neq}$.  
This proves \eqref{{2}f} already for a set $B$ and functions $u$ of the form \eqref{{2}B}. Thus, we obtain statement [\ref{IIIs}].

[\ref{IIIs}]$\overset{{\it \tiny proof}}{\Longrightarrow}$[\ref{IVs}].
 All functions $u_x$ in  [\ref{IVs}] are subharmonic on $\RR^d\supset O$.

[\ref{IVs}]$\overset{{\it \tiny proof}}{\Longrightarrow}$[\ref{Is}].
The Riesz measure of $u_x$ is the Dirac measure $\delta_x$, and, by  [\ref{IVs}], 
\begin{equation}\label{PJK}
\int_S u_x\dd \varDelta_q+\int_{B_j} p\dd \delta_x
\overset{\eqref{{2}f}}{=}\int_S u_x\dd \varDelta_p+\int_{B_j} q\dd \delta_x
\quad\text{for each $j\in \NN$ and $x\in O$}. 
\end{equation} 
If $j=0$ and $x\notin S_O=B_0$, then $\supp \delta_x=x\notin S_O$ and
\begin{equation*}
\int_{S_O} p\dd \delta_x\overset{\eqref{PJK}}{=}\int_{S_O} q\dd \delta_x=0
\end{equation*}
These equalities do not depend on $j\in \NN_0$ for points $x\notin S_O$. Hence 
\begin{equation*}
\int_S u_x\dd \varDelta_q\overset{\eqref{PJK}}{=}\int_S u_x\dd \varDelta_p
\quad\text{for each $j\in \NN_0$ and $x\notin S_O\supset S$}. 
\end{equation*}
Therefore, it is follows from \eqref{PJK} that 
\begin{equation*}
\int_{B_j} p\dd \delta_x
\overset{\eqref{PJK}}{=}\int_{B_j} q\dd \delta_x
\quad\text{for each $j\in \NN_0$ and $x\notin S_O$}, 
\end{equation*} 
i.e., $p(x)=q(x)$ for each $j\in \NN_0$ and for every  $x\in B_j\!\setminus\!S_O$. 
Thus, $p(x)=q(x)$  for each  point $x\in \bigcup_{j\in \NN_0}B_j\!\setminus\!S_O
=O\!\setminus\!S_O$, and statement [\ref{Is}] is established.

\section{Duality Theorems for balayage}\label{DT}
\setcounter{equation}{0}
Part of  some  equivalences of the Main Lemma  and the Main Theorem 
allows us to give an internal dual description for  the potentials of measures obtained through the balayage   processes.
Such descriptions in particular cases of Arens\,--\,Singer and Jensen measures 
and their potentials have already found important applications in the study of various problems of function theory \cite[Ch.~3 etc.]{Gamelin},  \cite{Anderson}, \cite{Kha91}, \cite{Kha92I}, \cite{Kha92II}, \cite{Kha93}, \cite{Kha94}, \cite{Kha94BM}, \cite{Koosis96}, \cite{Kha96}, \cite{Kh010}, \cite{Kha01}, \cite{Kha01II}, \cite{Ransford01}, \cite{Kha02It}, \cite{Kha03}, \cite{Kha07},  \cite{Kha12},  \cite{BaiTalKha16}, \cite{KhaRozKha19}, \cite{KhaTalKha15},  \cite{KhaKhaChe09}.

\begin{dualtheorem}[{\rm for $\har(O)$-balayage}]\label{DT1}
Let $\varDelta \in \Meas_{\comp}^+(O)$. 

If  a measure $\upomega \in \Meas_{\comp}^+(O)$ is a\/ $\har (O)$-ba\-la\-ya\-ge of  ${\varDelta}$,  then {\rm (cf. \eqref{ASpc})} 
\begin{subequations}\label{pmu0}
\begin{align}
{\pt}_{\upomega}&\in \sbh_*(\RR^d)\cap 
\har(\RR^d\!\setminus\! \supp \upomega), 
\tag{\ref{pmu0}p}\label{{pmu0}p}
\\ 
{\pt}_{\upomega}&=  {\pt}_{\varDelta} \text{ on $\RR^d \!\setminus\! \infill_O (\supp \varDelta  \cup \supp \upomega)$}.
\tag{\ref{pmu0}=}\label{{pmu0}o}
\end{align}
\end{subequations}

\underline{Conversely}, suppose that there are a compact subset $S\Subset O$
and  a function $p$ such that 
\begin{subequations}\label{p}
\begin{align}
p&\overset{\text{cf.}\eqref{{pmu0}p}}{\in} \sbh (O)\cap  \har (O\!\setminus\!S),
\tag{\ref{p}p}\label{{pmu0+}p}
\\
p&\overset{\text{cf.} \eqref{{pmu0}o}}{=} {\pt}_\varDelta \quad\text{on $O\!\setminus\!S$}.
\tag{\ref{p}=}\label{{pmu0+}o}
\end{align}
\end{subequations}
Then the  Riesz measure
\begin{equation}\label{mu}
\upomega\overset{\eqref{df:cm}}{:=}c_d\bigtriangleup\! p \overset{\eqref{{pmu0+}p}}{\in} \Meas^+(S)\subset \Meas_{\comp}^+(O)
\end{equation} 
of this function $p$ is a $\har(O)$-balayage of  the measure $\varDelta$. 
\end{dualtheorem}
\begin{proof}  Properties \eqref{pmu0} for $\upomega\succeq_{\har(O)}\varDelta$ directly follow  from the implication \ref{Is}$\Rightarrow$\ref{IIIs} of the Main Lemma. In the opposite direction, we can use  the implication \ref{IVm}$\Rightarrow$\ref{Im} of the Main Lemma with  $p$ from \eqref{p} and $q:=\pt_{\varDelta}$.
\end{proof}

\begin{dualtheorem}[{\rm for $\sbh(O)$-balayage}]\label{DTsbh}
Let $\varDelta \in \Meas^+_{\comp}(O)$. If $\upomega \succeq_{\sbh(O)}\varDelta$,  then we have \eqref{pmu0} and 
${\pt}_{\upomega}\geq {\pt}_{\varDelta}$ on\/ $\RR^d$.
\underline{Conversely}, suppose that there are a compact subset $S$ in $O$
containing $\supp \varDelta$, and  a function $p$ satisfying \eqref{p} such that\/
\begin{equation}\label{pSO}
p\geq {\pt}_{\varDelta} \quad\text{ on\/ $S_O:=\infill (S)$}. 
\end{equation}
Then the  Riesz measure \eqref{mu} of  
 this function $p$ is a $\sbh(O)$-balayage of the measure $\varDelta$.
\end{dualtheorem} 
\begin{proof}  If $\varDelta\preceq_{\sbh(O)}\upomega$,  
then $\varDelta\preceq_{\har(O)}\upomega$, which was noted earlier in \S\ref{Bal}\eqref{4}, and, by Duality Theorem \ref{DT1}, we obtain \eqref{pmu0}.  Besides,  functions  $y\underset{\text{\tiny $y\in \RR^d$}}{\longmapsto} K_{d-2}(y,x)$ are subharmonic on  $\RR^d$ for each $x\in \RR^d$, and   
\begin{equation*}
\pt_{\varDelta}(x)=\int K_{d-2}(y,x)\dd \varDelta(y)
\overset{\eqref{bhar0}}{\leq} \int K_{d-2}(y,x)\dd \upomega(y)=
\pt_{\upomega}(x)\quad\text{for each $x\in \RR^d$}.
\end{equation*} 

In the opposite direction,  we set $q:=\pt_{\varDelta}\in \sbh_*(\RR^d)\cap \har(O\!\setminus\!S)$.  By Duality Theorem \ref{DT1}, the  Riesz measure $\varDelta_p\overset{\eqref{mu}}{=}\upomega\in \Meas_{\comp}^+(O)$ of the function $p$ is a  $\har(O)$-balayage of $\varDelta$.
By condition \eqref{{pmu0+}o} in the notation \eqref{pSO},  we have the equality $p=q$ on $O\!\setminus\! S_O\subset O\!\setminus\!S$, and, 
by condition \eqref{{pmu0+}p},   the functions $p$ and $q$ are harmonic on $O\!\setminus\!S$. Thus, the  statement [\ref{Is}] of the Main Theorem is fulfilled.
By the implication  [\ref{Is}]$\Rightarrow$[\ref{IIIs}] of the Main Theorem, using the full symmetric Poisson\,--\,Jensen formula \eqref{{2}f} with  $B\overset{\eqref{{2}B}}{:=}S_O$, we get  
\begin{equation}\label{fS'}
\int_{S} u\dd \varDelta_q+\int_{S_O} p\dd \varDelta_u \overset{\eqref{{2}f}}{=}\int_{S} u\dd \varDelta_p+\int_{S_O} q\dd \varDelta_u
\quad \text{for each $u\overset{\eqref{{2}B}}{\in} \sbh_*(O )$}.
\end{equation}
Hence, by the condition $p\overset{\eqref{pSO}}{\geq} \pt_{\varDelta}=q$ on $S_O$, we obtain 
\begin{multline}\label{uupp}
\int_O u\dd \varDelta+\int_{S_O} q\dd \varDelta_u
=\int_S u\dd \varDelta_q+\int_{S_O} q\dd \varDelta_u
\leq \int_{S} u\dd \varDelta_q+\int_{S_O} p\dd \varDelta_u\\
\overset{\eqref{fS'}}{=} 
\int_{S} u\dd \varDelta_p+\int_{S_O} q\dd \varDelta_u
=\int_{O} u\dd\upomega+\int_{S_O} q\dd \varDelta_u
\quad \text{for each $u\in \sbh_*(O )$}.
\end{multline}
In particular, if $u\in \sbh_*(O)\cap C^{\infty}(O)$, then the function $q$ is $\varDelta_u$-integrable on $S_O$, and it is follows from  \eqref{uupp} that 
\begin{equation*}
\int_O u\dd \varDelta
\overset{\eqref{uupp}}{\leq} 
\int_{O} u\dd \upomega\quad \text{for each $u\in \sbh_*(O )\cap C^{\infty}(O)$}.
\end{equation*}
Hence, by  \S\ref{Bal}\eqref{pr:diff}, we obtain $\varDelta\preceq_{\sbh(O)}\upomega$.
\end{proof}
The following long-known result  for Arens\,--\,Singer and Jensen measures and their potentials on domains in $\RR^d$ with $d\geq 2$ 
has found numerous applications in the theory of functions of 
one and several complex variables,  which is partially reflected in the bibliographic sources listed at the beginning of  Sec.~\ref{DT}. The proof of this result 
 immediately follows from Duality Theorems \ref{DT1} and \ref{DTsbh}, but  already for open sets $O$ in $\RR^d$ with $d\in \NN$.

\begin{dualtheorem}[{\rm \cite[Proposition 1.4, Duality Theorem]{Kha03}}]\label{DT3} Let $x\in O\subset \RR^d$. The map 
\begin{equation}\label{mcP}
{\mathcal  P}_x \colon \upomega \overset{\eqref{pot}}{\longmapsto} {\pt}_{\upomega-\delta_x}
\end{equation}
defines an  affine bijection from $AS_x(O)$ onto $ASP_x(O)$, as well as  
from $J_x(O)$ onto $JP_x(O)$\/ {\rm (see also, in addition,  \eqref{ASpc})}
 with the inverse map
\begin{equation}\label{P-1}
\mathcal P_x^{-1} \colon V
\overset{\eqref{df:cm}}{\longmapsto} 
c_d {\bigtriangleup}V \bigm|_{\RR^d\!\setminus\!x}+\left(1-\limsup_{x\neq y\to x} \frac{V(y)}{-K_{d-2}(x,y)}\right)\cdot \delta_x.
\end{equation}
\end{dualtheorem}

\begin{remark}  Theorems \ref{DT1} and \ref{DTsbh} can also be formulated in a form close to Theorem 3, using some affine bijection of type \eqref{mcP}--\eqref{P-1} and definitions of the generalized Arens\,--\,Singer and Jensen potentials. But such formulations require some development of the theory of $\delta$-subharmonic functions \cite{Arsove}, \cite{Arsove53p}, \cite[3.1]{KhaRoz18} and a delicate approach to upper/lower integrals \eqref{int}
with values in $\overline \RR$. We will not discuss similar interpretations of Theorems  \ref{DT1} and \ref{DTsbh} here.
\end{remark}

\paragraph{\bf Acknowledgements} 
The work was supported by a Grant of the Russian Science Foundation (Project No. 18-11-00002).

{\sl Bashkir State University, Ufa, Bashkortostan, Russian Federation} 


\begin{thebibliography}{86}


\bibitem{Anderson}  Anderson, S. L.: Green's Function, Jensen Measures, and Bounded Point Evaluations. J. Func. Analysis, {\bf 43}, 360--367  (1981)

\bibitem{AG}  Armitage, D. H.,  Gardiner, S. J.: Classical potential theory. Springer Monogr. Math., Springer-Verlag, London (2001)

\bibitem{Arsove} Arsove, M.G.:  Functions representable as differences of subharmonic f\-u\-n\-c\-t\-i\-ons. Trans. Amer. Math. Soc. {\bf 75}, 327--365  (1953)

\bibitem{Arsove53p}  Arsove, M.G.: Functions of potential type. Trans. Amer. Math. Soc. {\bf 75}, 526--551 (1953)

\bibitem{ABR}
 Axler S., Bourdon P., Ramey W., Harmonic function theory, Second edition, Springer-Verlag, New York, 2001, 270 pp.

\bibitem{BagbyGauther} Bagby, Th., Guathier P.M.: Harmonic approximation on closed subsets of Riemannian manifolds. In book:
Complex Potential Theory. Netherlands: Kluwer Academic Publisher.
 75--87 (1994)

\bibitem{BaiTalKha16} 
Bayguskarov, T.Yu., Talipova, G.R., Khabibullin, B.N.:
 Subsequences of zeros for classes of entire functions of exponential type, allocated by restrictions on their growth, (Russian). Algebra i Analiz. {\bf 28}(2), 1--33 (2016);
English transl. in  St. Petersburg Math. J. {\bf 28}(2), 127--151  (2017)

\bibitem{BH} 
Bliedner, J., Hansen,  W.:  Potential Theory. An Analytic and Probabilistic Approach to Balayage, Springer-Verlag, Berlin (1986)


\bibitem{Bourbaki} Bourbaki,  N.: \'El\'ements de Math\'ematique. Livre~VI. Int\'egration (French).
Hermann, Paris (1969)

\bibitem{Schachermayer}  
Bu, S., Schachermayer, W.: Approximation of Jensen me\-a\-su\-res
by image   me\-a\-su\-res  under holomorphic functions and applications.
Trans. Amer. Math. Soc.  {\bf 331}(2), 585--608  (1992)

\bibitem{Chi18} 
Chirka, E.M.:  Potentials on a compact Riemann surface. (Russian)
In book: Complex analysis, mathematical physics, and applications.
Collected papers Tr. Mat. Inst. Steklova. vol. 301.
 MAIK Nauka/Interperiodica, Moscow, 287--319 (2018);
English transl. in Proc. Steklov Inst. Math. vol. 301, 272--303 (2018)

\bibitem{C-R} 
Cole, B.J., Ransford Th.,J:  Sub\-har\-mo\-ni\-ci\-ty wit\-hout
Up\-per Se\-mi\-co\-n\-ti\-nu\-i\-ty. J. Funct. Anal.  {\bf 147}, 420--442  (1997)

\bibitem{C-RJ} 
Cole, B.J., Ransford Th.,J: 
Jensen measures and harmonic measures. J. reine angew. Math. {\bf 541}, 29--53  (2001)

\bibitem{Doob}  Doob, J.L.:  Classical Potential Theory and Its Probabilistic
Counterpart. Grundlehren Math. Wiss. Springer-Verlag, New York (1984)

\bibitem{Gamelin}  Gamelin, T.W.:
 Uniform Algebras and Jensen Measures.  Cam\-b\-r\-i\-d\-ge Univ. Press, Cam\-b\-r\-i\-d\-ge (1978)

\bibitem{Gardiner}  Gardiner, S.J.: Harmonic Approximation.
Cambridge Univ. Press, Cam\-b\-r\-i\-d\-ge (1995)

\bibitem{Gauther_B} Gauthier, P.M.:  Uniform approximation. In book:
Complex Potential Theory. Netherlands: Kluwer Academic Publisher.
 235--271 (1994)

\bibitem{Gauther_C} Gautier, P.M.: Subharmonic extensions and approximations.
Can. Math. Bull. {\bf 37}(5), 46--53 (1994)

\bibitem{quad}  Ghergu, M., Manolaki, M., Netuka, I.,  Render, H.:
Potential theory and approximation: highlights from the scientific work of Stephen Gardiner.  Analysis and Mathematical Physics. {\bf 9}(2), 679--709 (2019)

\bibitem{HN11}
Hansen, W.,  Netuka, I.:  Jensen Measures in Potential Theory. Potential Analysis, {\bf 37}(1), 79--90 (2011).

\bibitem{HN} Hansen, W.,  Netuka, I.:
Reduced functions and Jensen measures.
 Proc.  Amer. Math. Soc. \textbf{146}(1), 153--160  (2018)

\bibitem{HK}  
Hayman, W.K.,  Kennedy, P.B.: Subharmonic functions.  vol. 1.  
Acad. Press, London etc. (1976)

\bibitem{Helms}
 Helms, L.L.: Introduction to Potential Theory, Wiley Interscience, 
New York London Sydney Toronto (1969)

\bibitem{Kha91}
 Khabibullin, B.N.: Sets of uniqueness in spaces of entire functions of a single variable. (Russian) Izv. Akad. Nauk SSSR Ser. Mat., {\bf 55}(5), 1101--1123  (1991); English transl. in Math. USSR-Izv., {\bf 39}(2), 1063--1084  (1992) 

\bibitem{Kha92I}
Khabibullin, B.N.: The least plurisuperharmonic majorant, and multipliers of entire functions. II. (Russian) Sibirsk. Mat. Zh. {\bf 33}(1), 173--178 (1992); English transl. in Siberian Math. J., {\bf 33}(1), 144--148 (1992)

\bibitem{Kha92II}
 Khabibullin, B.N.: The least plurisuperharmonic majorant, and multipliers of entire functions. II. Algebras of functions of finite $\lambda$-type. (Russian) Sibirsk. Mat. Zh., {\bf 33}(3), 186--191 (1992);  English transl. in Siberian Math. J., {\bf 33}(3), 519--524 (1992)

\bibitem{Kha93}
Khabibullin, B.N.: 
The theorem on the least majorant and its applications. I.~Entire and meromorphic functions. (Russian). Izv. RAN. Ser. Mat., {\bf 57}(1),  129--146 (1993);
English transl. in  Russian Acad. Sci. Izv. Math. {\bf 42}(1), 115--131 (1994)

\bibitem{Kha94} 
Khabibullin, B.N.:  The theorem on the least majorant and its applications. II. Entire and meromorphic functions of finite order. (Russian). Izv. RAN. Ser. matem. {\bf 57}(3), 70--91 (1993); English transl. in  Russian Acad. Sci. Izv. Math., {\bf 42}(3), 479--500  (1994)

\bibitem{Kha94BM} 
 Khabibullin, B.N.: Nonconstructive proofs of the Beurling\,--\,Malliavin theorem on the radius of completeness, and nonuniqueness theorems for entire functions. (Russian).
Izv. RAN. Ser. matem. {\bf 58}(4), 125--148 (1994); English transl. in  
 Russian Acad. Sci. Izv. Math., \textbf{45}(1), 125--149  (1995)

\bibitem{Kha96}
 Khabibullin, B.N.: Zero sets for classes of entire functions and a representation of meromorphic functions. (Russian). Matem. Zametki. {\bf 59}(4), 611--617 (1996); English transl. in  Math. Notes, \textbf{59}(4), 440--444  (1996)

\bibitem{Kh010}	Khabibullin, B.N.: Dual approach to certain questions for weighted spaces of holomorphic functions. Entire functions in modern analysis (Tel-Aviv, December 14--19, 1997), Israel Math. Conf. Proc., \textbf{15}, Bar-Ilan Univ., Ramat Gan. 207--219  (2001) 

\bibitem{Kha02It} Khabibullin, B.N.: Completeness of sets of complex exponentials in convex sets: open problems. Proceedings of the NATO Advanced Study Institute on Twentieth Century Harmonic Analysis --- A Celebration  (Il Ciocco, Italy, July 2-15, 2000), NATO Sci. Ser. II, Math. Phys. Chem. Printed in the Netherlands, {\bf 33}, eds. James S. Byrnes: Kluwer Acad. Publ., Dordrecht,  371--373 (2001)

\bibitem{Kha01}
 Khabibullin, B.N.: On the Rubel\,--\,Taylor Problem on a Representation of Holomorphic Functions. (Russian). Funktsional. Anal. i Prilozhen. {\bf 35}(3), 91--94 (2001);
English transl. in  Funct. Anal. Appl. {\bf 35}(3), 237--239 (2001)

\bibitem{Kha01II}
Khabibullin, B.N.: Dual representation of superlinear functionals and its applications in function theory. II. (Russian). Izv. RaN. Ser. matem.  \textbf{65}(5), 167--190  (2001); English transl. in
Izv. Math. \textbf{65}(5), 1017--1039  (2001)

\bibitem{Kha03} Khabibullin, B.N.:  Criteria for (sub-)harmonicity and continuation of (sub-)harmonic functions. (Russian).  Sibirsk. Mat. Zh. {\bf44}(4), 905--925 (2003); English transl. in 
 Siberian Math. J. {\bf 44}(4), 713--728 (2003)  

\bibitem{Kha07}
 Khabibullin, B.N.:  Zero sequences of holomorphic functions, representation of~meromorphic functions, and harmonic minorants, (Russian). Mat. Sb. \textbf{198}(2), 121--160 (2007);
English transl. in Sb. Math. \textbf{198}(2), 261--298  (2007)

\bibitem{Kha12}
Khabibullin, B.N.: Completeness of systems of exponentials and sets of uniqueness, (Russian). 4th ed., revised and enlarged, Bashkir State University Press, Ufa  (2012).

\bibitem{KhaKha19}
 Khabibullin, B.N., Khabibullin, F.B.:  On the Distribution of Zero Sets of Holomorphic Functions. III. Inversion Theorems, (Russian) Funktsional. Anal. i Prilozhen., \textbf{53}(2), 42--58  (2019); English transl. in Funct. Anal. Appl. \textbf{53}(2), 110--123 (2019)

\bibitem{KhaKhaChe09}
 Khabibullin, B.N., Khabibullin, F.B., Cherednikova L. Yu.: 
Zero subsequences for classes of holomorphic functions: stability and the entropy of arcwise connectedness. I, II,  (Russian). Algebra i Analiz \textbf{20}(1), 146--236  (2009);
English transl. in St. Petersburg Math. J., \textbf{20}(1), 101--162  (2009) 

\bibitem{KhaRoz18}
Khabibullin, B.N., Rozit, A.P.: On the Distribution of Zero Sets of Holomorphic Functions, (Russian). Funktsional. Anal. i Prilozhen., \textbf{52}(1), 26--42  (2018);
English transl. in Funct. Anal. Appl., \textbf{52}(1), 21--34  (2018) 

\bibitem{KhaRozKha19}
  Khabibullin, B.N., Rozit,  A.P.,  Khabibullina, E.B.: Order versions of the Hahn\,--\,Banach theorem and envelopes. II. Applications to the function theory, (Russian). Complex Analysis. Mathematical Physics, Itogi Nauki i Tekhniki. Ser. Sovrem. Mat. Pril. Temat. Obz., {\bf 162}, VINITI, Moscow,
 93--135 (2019)

\bibitem{KhaTalKha15}
Khabibullin, B.N., Talipova, G.R.,  Khabibullin, F.B.:
 Zero subsequences for Bernstein's spaces and the completeness of exponential systems in spaces of functions on an interval, (Russian). Algebra i Analiz. \textbf{26}(2), 185--215  (2014); English transl. in  St. Petersburg Math. J. \textbf{26}(2), 319--340  (2015)


\bibitem{Koosis} Koosis, P.: The logarithmic integral. II. Cambridge Univ. Press, Cambridge (1992)

\bibitem{Koosis96}   Koosis, P.: Le\c cons sur le th\'eor\`eme Beurling et Malliavin, (French).
 Les Publications CRM, Mont\-r\'e\-al (1996)

\bibitem{KudKha09}
  Kudasheva, E. G., Khabibullin, B.N.: 
 The distribution of the zeros of holomorphic functions of moderate growth in the unit disc and the representation of meromorphic functions there, (Russian).
Matem. Sb. \textbf{200}(9), 1353--1382  (2009); English transl. in  Sb. Math. \textbf{200}(9), 1353--1382  (2009)


\bibitem{Landkoff}
Landkof, N.S.: Foundations of modern potential theory. 
Grundlehren Math. Wiss., vol. 180,  Springer-Verlag, New York Heidelberg (1972)


\bibitem{MOS}
Matsaev, V., Ostrovski\u{\i}  I., Sodin, M.: Variations on the theme of Marcinkiewicz' inequality. J. Anal. Math. \textbf{86}(1), 289--317  (2002) 
	
\bibitem{MenKha19}
Menshikova, E.B.,  Khabibullin, B.N.: On the Distribution of Zero Sets of Holomorphic Functions. II,
(Russian). Funktsional. Anal. i Prilozhen.  \textbf{53}(1), 84--87  (2019); English transl. in
Funct. Anal. Appl. \textbf{53}(1), 65--68  (2019)  

\bibitem{Meyer}
Meyer, P.-A.:  Probability and Potentials, Blaisdell Publ. Co., Waltham, Mass.--Toronto--London, 1966 

\bibitem{R}  Ransford, Th.: Potential Theory in the Complex 
Plane.   Cambridge University Press, Cambridge (1995)

\bibitem{Ransford01}
Ransford, Th.J.: Jensen measures. Approximation, complex analysis and potential theory (Montr\'eal, QC, 2000), Kluwer, Dordrecht, 221--237 (2001)

\bibitem{Sarason}
 Sarason, D.: Representing Measures for $R(X)$ and Their Green's Functions. J. Func. Anal.
\textbf{7}, 359--385  (1971)

\end{thebibliography}
\end{document}